\documentclass[reqno,twoside,12pt]{amsart}
\usepackage{amssymb,amsfonts,amsthm,amsmath}
\usepackage{enumitem}
\usepackage[pagebackref=false]{hyperref}
\usepackage[hmargin=1in,vmargin=1in]{geometry}
\usepackage{cite}

\def\eqdef{\stackrel{\rm def}{=}}

\def\d{{\rm d}}
\def\ddt{\frac{\d}{\d t}}

\newcommand{\tnum}{\rm(\roman*)}
\newcommand{\rnum}{\rm(\alph*)}

\def\beq{\begin{equation}}
\def\eeq{\end{equation}}
\def\beqs{\begin{equation*}}
\def\eeqs{\end{equation*}}

\newtheorem{theorem}{Theorem}[section]
\newtheorem{lemma}[theorem]{Lemma}
\newtheorem{proposition}[theorem]{Proposition}

\newtheorem{definition}[theorem]{Definition}
\newtheorem{assumption}[theorem]{Assumption}

\theoremstyle{definition}
\newtheorem{remark}[theorem]{Remark}

\newtheorem*{xnotation}{Notation}

\usepackage{color}

\def\varep{\varepsilon}

\newcommand{\R}{\ensuremath{\mathbb R}}

\newcommand{\N}{\ensuremath{\mathbb N}}

\newcommand{\bigo}{\mathcal O}

\numberwithin{equation}{section}

\title{Behavior near the extinction time for systems of differential equations with sublinear dissipation terms}
\author{Luan Hoang}

\address{Department of Mathematics and Statistics,
Texas Tech University\\
1108 Memorial Circle, Lubbock, TX 79409--1042, U. S. A.}
\email{luan.hoang@ttu.edu}


\date{\today}

\begin{document}

\begin{abstract} 
This paper is focused on the behavior near the extinction time of solutions of systems of ordinary differential equations with a sublinear dissipation term. Suppose the dissipation term is a product of a linear mapping $A$ and a positively homogeneous scalar  function $H$ of a negative degree $-\alpha$.  Then any solution  with an extinction time $T_*$ behaves like $(T_*-t)^{1/\alpha}\xi_*$ as time $t\to T_*^-$, where $\xi_*$ is an eigenvector of $A$. The result allows the higher order terms to be general and the nonlinear function $H$ to take very complicated forms. As a demonstration, our theoretical study is applied to an inhomogeneous population model.
\end{abstract}

\maketitle

\tableofcontents

\pagestyle{myheadings}\markboth{L. Hoang}
{Behavior Near the Extinction Time for Differential Equations with Sublinear Dissipation Terms}


\section{Introduction}\label{intro}
This paper continues our investigations of exact asymptotic behaviors of solutions of systems of nonlinear  ordinary differential equations (ODE), see \cite{CaH3,CaHK1,H5,H7}, and the Navier--Stokes equations (NSE), see, e.g., \cite{HM2,HTi1,CaH2,H6}. 
However, in contrast to the cited work above, the current paper studies the asymptotic behavior near a finite extinction time, instead of time infinity, for equations with sublinear dissipation terms, instead of superlinear in \cite{H7}, or linear in the others.
The phenomena of extinction are widely studied in mathematics, biology and physics, see e.g.  \cite{CMS2023,DPV2016,BPR2009,MNS2023,Pino2002,SJC2023,KaKa2016,BioExtinct2009}. In biology, the extinction time reflects the time when the species' populations go extinct, while, for  fast diffusive fluid flows in porous media, it indicates the time when the pressure or density vanishes everywhere. 
For the precise description of the solutions near a finite extinction time, the mathematical results in \cite{CMS2023,MNS2023,Pino2002} are only for very specific equations. 
Instead, our aim is to establish  the results for general systems of ODEs. 
It turns out that we can achieve this by modifying and improving some techniques  by Foias and Saut in \cite{FS84a} for the NSE. 
For that reason, we  review  \cite{FS84a} and the related literature here. The NSE with potential body forces  can be written in the following functional form, which holds in a certain weak sense in a suitable functional space,
\beq\label{NSE}
u'+Au+B(u,u)=0,
\eeq
where $A$ is  the (linear) Stokes operator which has positive eigenvalues and $B(\cdot,\cdot)$ is a bilinear form. 
It is proved in \cite{FS84a} that any nontrivial solution  $u(t)$ of \eqref{NSE} has the following asymptotic behavior
\beq\label{FSlim}
e^{\Lambda t}u(t)\to \xi_* \text{ in any $C^m$-norms as $t\to\infty$,}
\eeq
where $\Lambda$  is an eigenvalue of $A$ and  $\xi_*$ is an eigenfunction of $A$ associated with $\Lambda$.
For finer asymptotic behaviors than \eqref{FSlim}, Foias and Saut develop a theory of  asymptotic expansions for the solutions of the NSE \eqref{NSE} in \cite{FS87,FS91}. 

On the one hand, the result \eqref{FSlim} and its proof  are extended to abstract differential inequalities in \cite{Ghidaglia1986b,Ghidaglia1986a}. 
On the other hand, the asymptotic expansion theory is developed further for both ODE and partial differential equations (PDE). See \cite{Minea,Shi2000,HM1,HTi1,CaHK1,H7} for systems  without forcing functions, \cite{HM2,CaH1,CaH2,CaH3,H5,H6} for systems with forcing functions, and  \cite{H4} for the Lagrangian trajectories for viscous incompressible fluids. 
Both techniques from \cite{FS84a} and \cite{FS87} are combined in \cite{CaHK1} to deal with nonsmooth ODE systems. Originally, the asymptotic expansions can be obtained independently from the limit \eqref{FSlim}. 
In  \cite{CaHK1}, however, they are obtained only after the first asymptotic approximation \eqref{FSlim} is established.

In the original NSE in \cite{FS84a} as well as the systems in the extended work cited above, except for\cite{H7}, the ODE or PDE have linear  dissipation terms.  On contrary, the author recently studied in \cite{H7} the following ODE system in $\R^n$
\beq\label{yReq}
y'=-H(y) Ay +G(t,y),
\eeq
where $A$ is a constant $n\times n$ matrix with positive eigenvalues, $H$  is a positive function, and $G$ represents a higher order term. 
In \cite{H7}, $H$  is additionally assumed to be a positively homogeneous function of a positive degree $\alpha$. Note that the dissipation term $H(y)Ay$ in \eqref{yReq} is nonlinear compared with, say, the linear  dissipation term $Au$ in \eqref{NSE}, and the higher order term $G$ is not required to be bilinear like $B(u,u)$.
It is proved that  any nonzero, decaying solution of \eqref{yReq} behaves like $t^{-1/\alpha}\xi_*$ as $t\to\infty$, where $\xi_*$ is an eigenvector of $A$.

The current paper considers  the opposite scenario when the function $H$ in \eqref{yReq} has a negative degree $-\alpha$. In this case, many solutions start out with nonzero values and then become zero at a finite time. Such time is called the \textit{extinction time}. Our goal is to describe the behavior of these solutions near this extinction time.
The main result can be briefly described as follows. Under appropriate assumptions, any  solution $y(t)$ of \eqref{yReq} with the extinction time $T_*$ behaves exactly like $(T_*-t)^{1/\alpha}\xi_*$ as time $t\to T_*^-$, where $\xi_*$ is an eigenvector of $A$. It is worth mentioning that the existence of the extinction time is guaranteed under the small nonzero initial data condition, see Theorem \ref{mustdie} below. Our proof will make use and adapt the techniques from \cite{FS84a,H7}. In particular, the recent perturbation method  in \cite{H7} will be utilized. This method is needed to deal with the nonlinear dissipation in our problem. It will be implemented successfully in this paper for the study of the asymptotic behavior near the finite extinction time, instead of at time infinity as in \cite{H7}. The obtained result, in addition to its merits for ODE, also gives hints to a type of results that may be expected for general nonlinear PDE of the similar structure.
 
The paper is organized as follows.
Section \ref{mainresult} contains the main results. 
While the condition for the matrix $A$ is the natural Assumption \ref{assumpA}, the more technical conditions for the function $H$ are specified in  Assumption \ref{Hcond}. A key requirement of $H$, namely, property (HC)   is introduced in Definition \ref{HCdef}.
Theorem  \ref{mustdie} states that under appropriate conditions on $A$, $H$ and $G$, for any sufficiently small nonzero initial condition, there exists a solution of \eqref{yReq}   
that will become zero at finite time. Solutions with a finite extinction time of this type are the objects of our investigation in this paper.
The asymptotic behavior of the solutions to a more general equation \eqref{mainode} near the extinction time is established in Theorem \ref{mainthm}. Its counterpart for equation \eqref{yReq} is Theorem \ref{mainthm2}.
The proof of Theorem  \ref{mustdie} is given in Section \ref{extime}.
Section \ref{facts} prepares for the proof of Theorem \ref{mainthm}. Preliminary estimates for the solutions are obtained in Lemma \ref{estthm}. Although they provide only a rough description of $y(t)$, the upper and lower bounds with the same rate $1/\alpha$ obtained in \eqref{yplus}  are important in our further analysis. 
Section \ref{simeg} proves Theorem \ref{mainthm} for a special case in the form of equation \eqref{basicf}.
This will also serve as the basis for the perturbation argument for the general case in Section \ref{gencase}.
In Section \ref{symcase}, we obtain essential properties of the solutions of \eqref{mainode} when the matrix $A$ is symmetric. 
 In particular, we establish an eigenvalue $\Lambda$ as the limit of the quotient $\lambda(t)$, see \eqref{quot},  in Proposition \ref{lem1}, and a unit vector $v_*$ as the limit of $y(t)/|y(t)|$ in Propositions \ref{lem2} and \ref{lem4}.
In Section \ref{gencase}, we give proof to Theorem \ref{mainthm} first.
It combines all the previous preparations with the perturbation method mentioned earlier, see equation \eqref{Rygood}. This equation is a reduction of equation \eqref{mainode} of $y(t)$ to the simple form \eqref{basicf}, but for the projection $R_\Lambda y(t)$ and with a frozen coefficient $\Lambda H(v_*)$. 
The proof of Theorem \ref{mainthm2} is then quickly provided.
Section \ref{examples} contains some examples for the function $H$ in subsection \ref{egs}, and an application to an inhomogeneous population model in subsection \ref{bioeg}.

\begin{xnotation}
Throughout the paper, $n\in\N=\{1,2,3,\ldots\}$ is the spatial dimension.
For any vector $x\in\R^n$, we denote by $|x|$ its Euclidean norm.
For an $n\times n$ real matrix $A=(a_{ij})_{1\le i,j\le n}$, its Euclidean norm is 
$$\|A\|=\left(\sum_{i=1}^n\sum_{j=1}^n a_{ij}^2\right)^{1/2}.$$
The unit sphere in $\R^n$ is $\mathbb S^{n-1}=\{x\in\R^n:|x|=1\}$.
\end{xnotation}

\section{Main results}\label{mainresult}

We present the main results of the paper in this section.

\begin{assumption}\label{assumpA}
Matrix  $A$ is a (real) diagonalizable $n\times n$ matrix with positive eigenvalues.
\end{assumption}
 
Under Assumption \ref{assumpA}, the matrix $A$ has $n$ positive eigenvalues (counting their multiplicities)
$\Lambda_1\le \Lambda_2\le \Lambda_3\le \ldots\le \Lambda_n$,
and there exists an invertible $n\times n$ (real) matrix $S$ such that
\beq \label{Adiag}
A=S^{-1} A_0 S,
\text{ where } A_0={\rm diag}[\Lambda_1,\Lambda_2,\ldots,\Lambda_n].
\eeq 
In the case $A$ is symmetric, the matrix $S$ is orthogonal, i.e., $S^{-1}=S^{\rm T}$, and
\beq\label{xAx}
\Lambda_1|x|^2\le x\cdot Ax\le \Lambda_n |x|^2 \text{ for all $x\in\R^n$.}
\eeq

More specifically, the distinct eigenvalues of $A$ are denoted by $\lambda_j$, with $1\le j\le d$  for some integer $d\in[1,n]$, and are arranged to be (strictly) increasing in $j$, i.e., 
 \beqs
 0<\lambda_1=\Lambda_1<\lambda_2<\ldots<\lambda_{d}=\Lambda_n.
 \eeqs
The spectrum  of $A$ is $\sigma(A)=\{\Lambda_k:1\le k\le n\}=\{\lambda_j:1\le j\le d\}$.

For $1\le k,\ell\le n$, let $E_{k\ell}$ be the elementary $n\times n$ matrix $(\delta_{ki}\delta_{\ell j})_{1\le i,j\le n}$, where $\delta_{ki}$ and $\delta_{\ell j}$ are the Kronecker delta symbols.
For $\Lambda\in\sigma(A)$, define
 \beqs
\widehat R_\Lambda=\sum_{1\le i\le n,\Lambda_i=\Lambda}E_{ii}\text{ and } R_\Lambda=S^{-1}\widehat R_\Lambda S.
 \eeqs
Then one immediately has
\beq \label{Pc} I_n = \sum_{j=1}^{d} R_{\lambda_j},
\quad R_{\lambda_i}R_{\lambda_j}=\delta_{ij}R_{\lambda_j},
\quad  AR_{\lambda_j}=R_{\lambda_j} A=\lambda_j R_{\lambda_j}.
\eeq 
Thanks to \eqref{Pc}, each $R_\Lambda$ is a projection, and $R_\Lambda(\R^n)$ is the eigenspace of $A$ associated with the eigenvalue $\Lambda$.

In the case $A$ is symmetric,  $R_\Lambda$ is  the orthogonal projection from $\R^n$ to  the eigenspace of $A$ associated with $\Lambda$, and, hence,
\beq \label{Rcontract}
|R_\Lambda x|\le |x| \text{ for  all $x\in\R^n$.}
\eeq

The function $H$ will be assumed to have some type of homogeneity which is specified in the next definition.

\begin{definition}\label{phom} 
Let $X$ and $Y$ be two (real) linear spaces, and $\beta<0$ be a given number. 

A function $F:X\setminus\{0\}\to Y$ is positively homogeneous of degree $\beta$ if
\beqs
F(tx)=t^\beta F(x)\text{ for any $x\in X\setminus\{0\}$ and $t>0$.}
\eeqs

Define $\mathcal H_\beta(X,Y)$ to be the set of functions  from $X\setminus\{0\}$ to $Y$ that are positively homogeneous  of degree $\beta$.
\end{definition}

If $F\in \mathcal H_\beta(X,Y)$ and $F$ is not the zero function, then the degree $\beta$ is unique.

 \begin{assumption}\label{assumpH}
The function $H$ is in $\mathcal H_{-\alpha}(\R^n,\R)$ for some $\alpha>0$, and in $C(\R^n\setminus\{0\},(0,\infty))$.
\end{assumption}

For a function $H$ in Assumption \ref{assumpH}, it is positive and continuous on $\mathbb S^{n-1}$. Hence, we have 
\beq\label{Hmm}
0<c_1=\min_{|x|=1} H(x)\le \max_{|x|=1} H(x) =c_2<\infty.
\eeq 
By writing $ H(x)=|x|^{-\alpha} H(x/|x|)$ for any $x\in\R^n\setminus\{0\}$ and using \eqref{Hmm}, we derive
\beq\label{Hyal}
c_1 |x|^{-\alpha} \le H(x)\le c_2 |x|^{-\alpha} \text{ for all }x\in\R^n\setminus \{0\}.
\eeq

Regarding the function $G$ in equation \eqref{yReq}, we have the following assumption.

 \begin{assumption}\label{assumpG}
Let $t_0$ be any given number in $[0,\infty)$. We assume that the function $G(t,x)$ is continuous on $[t_0,\infty)\times (\R^n\setminus\{0\})$, and there exist positive numbers $c_*, r_*, \delta$ such that 
\beq \label{Fcond}
 |G(t,x)| \le c_*|x|^{1-\alpha+\delta} \text{ for all  $t \ge t_0$, and all $x \in \R^n$  with $0<|x|\le r_*$.}
\eeq  
 \end{assumption}

In our first theorem below, we show that the extinction time exists for, at least, certain small solutions of  \eqref{yReq}. 

\begin{theorem}\label{mustdie}
Under Assumptions \ref{assumpA}, \ref{assumpH} and \ref{assumpG}, there exists a number $r_0>0$ such that for any $y_0\in \R^n\setminus\{0\}$ with $|y_0|\le r_0$, there are a number $T_*>t_0$ and a function 
$y\in C^1([t_0,T_*),\R^n)$ such that $y(t_0)=y_0$,
\beq\label{nz}
y(t)\ne 0\text{ for all $t\in [t_0,T_*)$, }
\eeq 
\beq \label{limzero}
\lim_{t\to T_*^-} y(t)= 0,
\eeq 
and $y(t)$ satisfies equation \eqref{yReq} for all  $t \in (t_0,T_*)$.
In other words, $y(t)$ is a solution of \eqref{yReq} on $(t_0,T_*)$ with the initial data $y_0$ at time $t_0$ and has the extinction time  $T_*$.
\end{theorem}

In fact, $y(t)$ is any solution  with the maximal interval $[t_0,T_{\max})$ in the existence theorem \ref{existsoln} below, and $T_*=T_{\max}$. 
Moreover, if additional conditions are imposed to guarantee the uniqueness of solutions to the initial value problem for equation \eqref{yReq}, then, in the above Theorem \ref{mustdie}, \emph{any} solution with sufficiently small nonzero initial data must have a finite extinction time.
 The proof of Theorem \ref{mustdie} is given in Section \ref{extime}.

For the behavior of a solution near an extinction time, the function $H$ is required to have an extra property. 

\begin{definition}\label{HCdef}
Let $E$ be a nonempty subset of $\R^n$ and $F$ be a function from $E$ to $\R$. We say $F$ has property (HC) on $E$ if, 
for any  $x_0\in E$,  there exist  numbers  $r,C,\gamma>0$ such that
\beq\label{FHder}
|F(x)-F(x_0)|\le C|x-x_0|^\gamma
\eeq
for any $x\in E$ with $|x-x_0|<r$.
\end{definition}

The following are elementary properties of the functions in Definitions \ref{phom} and \ref{HCdef}.

\begin{lemma}\label{Hcts}
Let $F\in \mathcal H_{-\alpha}(\R^n,\R)$ for some $\alpha>0$.
\begin{enumerate}[label=\tnum]
\item\label{HH1} If $F>0$ on $\mathbb S^{n-1}$, then $F>0$ on  $\R^n\setminus \{0\}$.

\item\label{HH2} If $F$ is continuous on  $\mathbb S^{n-1}$, then  it is continuous on $\R^n\setminus \{0\}$.

Assume $F$ has property (HC) on  $\mathbb S^{n-1}$ in \ref{HH3}--\ref{HH5} below.

\item\label{HH3} Then $F$ has property (HC) on $\R^n\setminus\{0\}$.

\item\label{HH4} If  $\varphi$ is a function from $\R^n\setminus\{0\}$ to $\R^n\setminus\{0\}$ that has property (HC) on $\R^n\setminus\{0\}$. Then $F\circ \varphi$ has property (HC) on $\R^n\setminus \{0\}$.

\item\label{HH5}  If $K$ is an invertible $n\times n$ matrix, then the function $x\in \R^n\setminus\{0\} \mapsto F(Kx)$ has property (HC) on $\R^n\setminus \{0\}$.
\end{enumerate}
\end{lemma}

The proof of Lemma \ref{Hcts} is given in the Appendix.

\begin{assumption}\label{Hcond}
The function $H$ belongs to $\mathcal H_{-\alpha}(\R^n,\R)$ for some $\alpha>0$, has property (HC) on the unit sphere  $\mathbb S^{n-1}$, and $H> 0$ on  $\mathbb S^{n-1}$.
\end{assumption}

If the function $H$ satisfies Assumption \ref{Hcond}, it is obvious that $H$ is continuous on $\mathbb S^{n-1}$. 
Therefore, thanks to  parts \ref{HH1} and \ref{HH2} of Lemma \ref{Hcts}, it is continuous and positive on $\R^n\setminus\{0\}$. Consequently, $H$ satisfies the conditions in Assumption \ref{assumpH}.
Some examples for the function $H$ will be  given in subsection \ref{egs}. 

The next result deals with a more general equation than \eqref{yReq}, namely, equation \eqref{mainode} below.

\begin{theorem}[Main Theorem I]\label{mainthm}
Let Assumptions \ref{assumpA} and  \ref{Hcond} hold.
Let $t_0,T_*\in\R$ be two given numbers with $T_*>t_0\ge 0$.
Assume $y\in C^1([t_0,T_*),\R^n)$ satisfies \eqref{nz}, \eqref{limzero}
and
\beq\label{mainode}
y'=-H(y)Ay +f(t) \text{ for all  } t \in (t_0,T_*),
\eeq
where  $f$ is a continuous function from $ [t_0,T_*)$ to $\R^n$ such that
\beq \label{frate}
|f(t)|\le M|y(t)|^{1-\alpha+\delta},\text{ for all $t\in [t_0,T_*)$ and some constants $M,\delta>0$.}
\eeq 
Then there exist an eigenvalue $\Lambda$ of $A$ and  an eigenvector $\xi_*$ of $A$ associated with $\Lambda$ such that
\beq\label{mainest}
|y(t)- (T_*-t)^{1/\alpha}\xi_*|=\bigo((T_*-t)^{1/\alpha+\varep})\text{ as $t\to T_*^-$ for some $\varep>0$.}
\eeq
More specifically, 
\beq\label{newIR}
|(I_n-R_\Lambda)y(t)|=\bigo((T_*-t)^{1/\alpha+\varep})\text{ as $t\to T_*^-$ for some $\varep>0$,}
\eeq
\beq\label{newRy}
| R_\Lambda y(t)- (T_*-t)^{1/\alpha}\xi_*|=\bigo((T_*-t)^{1/\alpha+\varep})\text{ as $t\to T_*^-$ for some $\varep>0$,}
\eeq
 and 
\beq\label{xiHA}
\alpha \Lambda H(\xi_*)=1.
\eeq
\end{theorem}

With a solution $y(t)$ as in Theorem \ref{mainthm}, we define,  for the sake of convenience, 
\beq \label{yTz}
y(T_*)=0, \text{ and have } y\in C([t_0,T_*],\R^n). 
\eeq 

From Theorem \ref{mainthm}, we can derive a corresponding result for equation \eqref{yReq}.

\begin{theorem}[Main Theorem II]\label{mainthm2}
Let Assumptions \ref{assumpA}, \ref{assumpG} (with a number $t_0\ge0$) and  \ref{Hcond} hold.
Given a number $T_*>t_0$. Assume $y\in C^1([t_0,T_*),\R^n)$ has properties \eqref{nz}, \eqref{limzero}, and
satisfies equation \eqref{yReq} for all $ t \in (t_0,T_*)$.
Then there exist an eigenvalue $\Lambda$ and an associated eigenvector $\xi_*$ of $A$ such that \eqref{mainest}--\eqref{xiHA} hold true.
\end{theorem}

The proofs of Theorems \ref{mainthm} and \ref{mainthm2} are given in Section \ref{gencase}.

\section{Existence of the extinction time}\label{extime}

We prove Theorem \ref{mustdie} in this section. 
First, we present a standard existence theorem for the solutions.

\begin{theorem}\label{existsoln}
Let Assumptions \ref{assumpA}, \ref{assumpH}  hold and assume $G$ is a continuous function from 
$[t_0,\infty) \times \R^n\setminus\{0\}$ to $\R^n$ for some number $t_0\in \R$.
Let  $y_0\in \R^n\setminus\{0\}$.  Then there exist an interval $[t_0,T_{\max})$, with $t_0<T_{\max}\le \infty$, and a function $y\in C^1([t_0,T_{\max}),\R^n\setminus\{0\})$ such that
\beq\label{ysat}
\text{ $y(t)$ that satisfies \eqref{yReq} on $(t_0,T_{\max})$, $y(t_0)=y_0$,}
\eeq  
and either
\begin{enumerate}[label=\rnum]
\item $T_{\max}=\infty$, or
\item\label{finT}  $T_{\max}<\infty$, and for any $\varep>0$, $y$ cannot be extended to a function of class $C^1([t_0,T_{\max}+\varep),\R^n\setminus\{0\})$ that satisties   \eqref{yReq} on the interval $(t_0,T_{\max}+\varep)$.
\end{enumerate}
Moreover, in the case \ref{finT}, it holds, for any compact set $U\subset \R^n\setminus\{0\}$, that
\beq\label{escape}
y(t)\not\in U\text{ when $t\in[t_0, T_{\max})$ is near $T_{\max}$.}
\eeq
\end{theorem}
\begin{proof}
 By  Peano's Existence Theorem \cite[Chapter II, Theorems 2.1, page 10]{Hartman1964} and the continuity of the function $G(t,x)$, there exist a number $\delta>0$ and a function $y\in C^1([t_0,t_0+\delta],\R^n\setminus\{0\})$ such $y(t)$  that satisfies equation  \eqref{yReq} on $[t_0,t_0+\delta]$ and $y(t_0)=y_0$.
By the virtue of the Extension Theorem, see \cite[Chapter II, Theorem 3.1, page 12]{Hartman1964} or  \cite[Chapter I, Theorem 2.1, page 17]{JHale78}, applied to this solution $y$ and the open set 
\beq\label{Dset} 
D=\left\{(t,x):t > t_0,x\in \R^n\setminus\{0\}\right \}\subset \R^{n+1},
\eeq
the current solution $y$ can be extended to a solution  $y\in C^1([t_0,T_{\max}),\R^n\setminus\{0\})$, for some $t_0+\delta\le T_{\max}\le \infty$, that has the  properties \eqref{ysat}, either (a) or (b), and, additionally, in the case (b), one has,
for any compact set $V\subset D$, 
\beq\label{escape2}
(t,y(t))\not\in V\text{ when $t\in[t_0, T_{\max})$ is near $T_{\max}$.}
\eeq

Now, consider case (b) and a compact set $U\subset \R^n\setminus\{0\}$. Let $V=[T_1,T_2]\times U$ with $T_1=(t_0+T_{\max})/2$ and $T_2=T_{\max}+1$. If  $t\in[t_0, T_{\max})$ is sufficiently close to $T_{\max}$, then $t\in [T_1,T_2]$. Therefore,  the desired statement \eqref{escape} follows from \eqref{escape2}.
\end{proof}

Note that such a solution $y(t)$ in Theorem \ref{existsoln} may not be unique. Next, we prove Theorem \ref{mustdie}.

\begin{proof}[Proof of Theorem \ref{mustdie}]
Let $y(t)$ be a solution of \eqref{yReq} as in Theorem \ref{existsoln} with the maximal interval $[t_0,T_{\max})$. Then 
\beq\label{nzmax}
y(t)\ne 0\text{ for all }t\in[t_0,T_{\max}).
\eeq
On $(t_0,T_{\max})$, we have
\beq\label{dya1}
\ddt (|y|^{\alpha})
=\alpha |y|^{\alpha-2}y'\cdot y
= \alpha \left(- |y|^{\alpha-2} H(y)(Ay)\cdot y + |y|^{\alpha-2} G(t,y)\cdot y\right).
\eeq

\medskip
\noindent\textit{Step 1.} Consider $A$ is symmetric first. Take $r_0>0$ such that 
$$2r_0\le r_*\text{ and } c_* (2r_0)^\delta \le  a_0\eqdef c_1\Lambda_1/2.$$

For $t>t_0$ sufficiently close to $t_0$, we have $|y(t)|<2r_0$. 
Let $[t_0,T)$ be the maximal interval in $[t_0,T_{\max})$ on which $|y(t)|<2r_0$.

Suppose $T<T_{\max}$. On the one hand, it must hold that
\beq\label{yT} 
|y(T)|=2r_0.\eeq 
 On the other hand, combining  \eqref{dya1} with \eqref{xAx}, \eqref{Hyal} and \eqref{Fcond},  we have, for $t\in(t_0,T)$, that
\beq\label{da0}
\ddt (|y|^{\alpha})
\le \alpha (- c_1\Lambda_1  + c_*|y|^\delta )\le \alpha (- c_1\Lambda_1  + c_*(2r_0)^\delta ) \le -\alpha a_0<0.
\eeq
Thus, $|y(T)|^\alpha\le |y_0|^\alpha$, which implies $|y(T)|\le |y_0|\le r_0$. This contradicts \eqref{yT}. Therefore, $T=T_{\max}$. 
For $t\in(t_0,T_{\max})$, integrating \eqref{da0} from $t_0$ to $t$ gives
\beq\label{ydies}
|y(t)|^\alpha\le |y_0|^\alpha-\alpha a_0 (t-t_0)  \text{ for all $t\in[t_0,T_{\max})$.}
\eeq

\medskip
\noindent\textit{Step 2.}
Consider the general matrix $A$. Using the equivalence \eqref{Adiag}, we set 
\beq \label{zSy}
z(t)=Sy(t)\text{ and }z_0=z(t_0)=Sy_0. 
\eeq 
Then
\beq\label{zeq0}
z'=-\widetilde H(z)A_0z+\widetilde G(t,z)\text{ for } t\in(t_0,T_{\max}),
\eeq
where 
\beqs
\widetilde H(z)=H(S^{-1}z),\quad 
\widetilde G(t,z)=SG(t,S^{-1}z) \text{ for $t\in[t_0,T_{\max})$ and $z\in\R^n\setminus\{0\}$.}
\eeqs

Note that
\beq\label{SSx}
\|S^{-1}\|^{-1}\cdot |x|\le |Sx|\le \|S\|\cdot|x| \text{ for all $x\in\R^n$.}
\eeq

Clearly, $\widetilde H$ satisfies the same condition as $H$ in Assumption \ref{assumpH}. 
Moreover, $\widetilde G(t,z)$ is continuous on $[t_0,\infty)\times (\R^n\setminus\{0\})$.
For $t\in[t_0,T_{\max})$ and $0<|z|\le r_*/\|S^{-1}\|$, we have $0<|S^{-1}z|\le r_*$, and then, by \eqref{Fcond} and \eqref{SSx},
\beqs
|\widetilde G(t,z)|\le \|S\| \cdot  c_*|S^{-1}z|^{1-\alpha+\delta}
\le c_* \|S\| \cdot 
\begin{cases}
\|S^{-1}\|^{1-\alpha+\delta} |z|^{1-\alpha+\delta},&\text{ if $1-\alpha+\delta\ge 0$,}\\
\|S\|^{-(1-\alpha+\delta)} |z|^{1-\alpha+\delta},&\text{ otherwise.}
\end{cases}
\eeqs

We apply the calculations in Step 1 to the solution $z(t)$ of \eqref{zeq0}. When $|y_0|>0$ is sufficiently small, we have $|z_0|>0$ is sufficiently small, and hence, similar to  estimate \eqref{ydies}, 
\beqs
|z(t)|^\alpha\le |z_0|^\alpha-\alpha \widetilde a_0 (t-t_0), \text{ for all $t\in[t_0,T_{\max})$ and some constant $ \widetilde a_0>0$.}
\eeqs
Therefore, 
\beq\label{yagain}
|y(t)|^\alpha\le \|S^{-1}\|^\alpha |z(t)|^\alpha 
\le \|S^{-1}\|^\alpha(|Sy_0|^\alpha-\alpha \widetilde a_0 (t-t_0)) \text{ for all $t\in[t_0,T_{\max})$.}
\eeq

\medskip
\noindent\textit{Step 3.}
If $T_{\max}=\infty$, then \eqref{yagain} implies that $|y(t)|^\alpha<0 $ for $t> t_0+|Sy_0|^\alpha/(\alpha \widetilde a_0)$, which is an obvious contradiction. Therefore,  $T_{\max}<\infty$. 
As a consequence of \eqref{yagain},  
\beq\label{yunib} 
|y(t)|\le R_0 \text{ on $[t_0,T_{\max})$, where $R_0=\|S^{-1}\|\cdot |Sy_0|>0$.}
\eeq

\medskip
\noindent\textit{Step 4.}
Let $T_*=T_{\max}$. Then \eqref{nzmax} implies \eqref{nz}.
For any $\varep>0$, let $U=\{x\in\R^n:  \varep\le |x|\le 2R_0\}$ in \eqref{escape}.
Taking into account \eqref{yunib}, one must have $|y(t)|<\varep$ when $t\in[t_0, T_*)$ is near $T_*$.
This proves  the zero limit in \eqref{limzero}.
\end{proof}

We remark that $y(t)$ may be zero for $t$ larger than the above  $T_{\max}$. However, this is excluded from our consideration of the set $D$ in \eqref{Dset}.  The reason is our sole focus on the finite extinction time and the solution before that time. 

\section{Preliminary estimates}\label{facts}
In this section, we prepare for the proof of Theorem \ref{mainthm} by obtaining preliminary estimates for the solution $y(t)$ of equation \eqref{mainode}. They even hold under a weaker condition than Assumption \ref{Hcond}. 

\begin{lemma}\label{estthm}
Let Assumptions \ref{assumpA} and \ref{assumpH} hold.
Given numbers $T_*>t_0\ge 0$, let $y\in C^1([t_0,T_*),\R^n)$ satisfy \eqref{nz}, \eqref{limzero}, \eqref{mainode}--\eqref{frate}.
Then there are positive constants  $C_1$ and $C_2$ such that 
 \beq\label{yplus}
 C_1(T_*-t)^{1/\alpha}
\le  |y(t)|\le 
 C_2(T_*-t)^{1/\alpha} \text{ for all $t\in[t_0,T_*]$.}
\eeq
 \end{lemma}
\begin{proof}
We prove \eqref{yplus}  for the case the matrix $A$  is symmetric first and then for $A$ not symmetric.

\medskip
\noindent\textit{Case 1.} Consider  $A$ is symmetric. 
For $t\in(t_0,T_*)$, we calculate, similarly to \eqref{dya1},
\beq\label{dy2}
\ddt (|y|^{\alpha})
= \alpha \left(- |y|^{\alpha-2} H(y)(Ay)\cdot y + |y|^{\alpha-2} f(t)\cdot y\right).
\eeq
Utilizing  \eqref{xAx}, \eqref{Hyal} and \eqref{frate}, one has 
\beqs
\alpha( -c_2 \Lambda_n  - M |y|^\delta)\le \ddt  (|y|^{\alpha})\le  \alpha( - c_1 \Lambda_1  +M|y|^\delta).
\eeqs

Let $a_1= c_1\Lambda_1/2$ and $a_2=c_2\Lambda_n+1$. Let $r_0>0$ be such that 
$M r_0^\delta= \min\{1, c_1\Lambda_1/2\}$. 

Thanks to \eqref{limzero}, there is $T\in(t_0,T_*)$ such  that $|y(t)|\le r_0$ on $(T,T_*)$. Hence, 
\beq\label{douineq} 
-\alpha a_2 \le \ddt(|y|^{\alpha})\le -\alpha a_1 \text{ on $(T,T_*)$.}
\eeq

For $t\in[T,T_*)$, integrating \eqref{douineq} from $t$ to $t'\in(t,T_*)$, passing to the limit $t'\to T_*^-$,  and using \eqref{limzero}, we obtain
\beq
 \label{yuplow}
\alpha a_1 (T_*-t)\le  |y(t)|^\alpha \le \alpha a_2 (T_*-t) \text{ for all $t\in [T,T_*]$.}
 \eeq
 Above, \eqref{yuplow} holds for $t=T_*$ thanks to \eqref{yTz}.
Note also that 
\beq\label{yTT}
0<a_3\eqdef \min_{t\in[t_0,T]} (T_*-t)^{-1/\alpha}|y(t)| \le a_4\eqdef \max_{t\in[t_0,T]} (T_*-t)^{-1/\alpha}|y(t)|<\infty.
\eeq
Combining \eqref{yuplow} with \eqref{yTT}, we obtain the desired estimates in \eqref{yplus} 
with 
$$C_1=\min\{(\alpha a_1)^{1/\alpha},a_3\}\text{ and }
C_2=\max\{(\alpha a_2)^{1/\alpha},a_4\}.$$

\medskip
\noindent\textit{Case 2.} Consider $A$ is not symmetric. Let $A=S^{-1}A_0 S$ as in \eqref{Adiag}.
Same as \eqref{zSy}, we set $z(t)=Sy(t)$ for $t\in[t_0,T_*]$.  
Then $z$ belongs to $C^1([t_0,T_*),\R^n\setminus\{0\})\cap C([t_0,T_*],\R^n)$, 
$z(t)\ne 0$ for all $t\in[t_0,T_*)$, $z(T_*)=0$, and 
\beq\label{zeq}
z'=-\widetilde H(z)A_0z+\widetilde{f}(t)\text{ for } t\in(t_0,T_*),
\eeq
where 
\beq\label{HRz}
\widetilde H(z)=H(S^{-1}z) \text{ for $z\in\R^n\setminus\{0\}$, and }
\widetilde {f}(t)=Sf(t) \text{ for $t\in[t_0,T_*)$.}
\eeq

One can verify, thanks to Assumption  \ref{assumpH}, that $\widetilde H$ belongs to $\mathcal H_{-\alpha}(\R^n,\R)$, and is positive and continuous on $\R^n\setminus\{0\}$.
Moreover, it is clear that the function $\widetilde f$ is continuous on $[t_0,T_*)$. Thanks to \eqref{frate} and \eqref{SSx}, it satisfies, for $t\in[t_0,T_*)$, 
\beq\label{fzMtil}
|\widetilde f(t)|\le \|S\|\cdot |f(t)|\le \|S\|  M |y(t)|^{1-\alpha+\delta}\le \widetilde M |z(t)|^{1-\alpha+\delta}, \\
\eeq
where
\beqs 
\widetilde M= 
\begin{cases} M \|S\|\cdot   \|S^{-1}\|^{1-\alpha+\delta},&\text{ if $1-\alpha+\delta\ge 0$,}\\
                      M \|S\|\cdot  \|S\|^{-(1-\alpha+\delta)}=M\|S\|^{\alpha-\delta},     &\text{ otherwise.}
\end{cases}
\eeqs
Therefore, we can apply the result in Case 1 to  the solution $z(t)$ and equation \eqref{zeq}. Then there exist two positive constants $C_1'$ and $C_2'$ such that 
\beq\label{zpower}
C_1'(T_*-t)^{1/\alpha}
\le  |z(t)|\le 
 C_2'(T_*-t)^{1/\alpha}
\text { for all $t\in [t_0,T_*]$.}
 \eeq
Combining \eqref{zpower} with the relations in \eqref{SSx}, we obtain the estimates in \eqref{yplus} for $y(t)$. 
\end{proof}

The following are two immediate consequences of Lemma \ref{estthm}.
\begin{enumerate}[label=\rnum]
\item By \eqref{Hyal} and \eqref{yplus}, we have, for all $t\in[t_0,T_*)$,
\beq\label{Hyt}
C_3 (T_*-t)^{-1}\le H(y(t)) \le C_4 (T_*-t)^{-1}, \text{ where $C_3=c_1C_2^{-\alpha}$ and $C_4=c_2C_1^{-\alpha}$.} 
\eeq

\item We also observe from \eqref{frate} and \eqref{yplus} that, for all $t\in [t_0,T_*)$,
\beq\label{fbound}
|f(t)| \le M(T_*-t)^{1/\alpha-1+\delta/\alpha}\cdot
\begin{cases}
C_2^{1-\alpha+\delta},&\text{ if $1-\alpha+\delta\ge 0$,}\\
C_1^{1-\alpha+\delta},&\text{ otherwise.}
\end{cases}
\eeq
\end{enumerate}

\section{Proof for a special case}\label{simeg}

Let $a$ be an arbitrarily  positive number. Consider equation \eqref{mainode} in the case 
\beq \label{AHsim}
A=I_n \text{ and } H(x)=a|x|^{-\alpha},
\eeq 
 that is, equation \eqref{mainode} becomes 
\beq\label{basicf}
y'=-a|y|^{-\alpha} y +f(t) \text{ for } t \in (t_0,T_*).
\eeq

Theorem \ref{mainthm} for this particular case is simply the following.

\begin{theorem}\label{simthm}
Given  numbers $T_*>t_0\ge 0$. Let $y\in C^1([t_0,T_*),\R^n)$ and $f\in C([t_0,T_*),\R^n)$ satisfy \eqref{nz}, \eqref{limzero}, \eqref{frate} and \eqref{basicf}.
Then there exists a  vector $\xi_*\in\R^n$ such that
\beq\label{axi1}
 |\xi_*|=(\alpha a)^{1/\alpha},
\eeq
and, as $t\to T_*^-$, 
\beq\label{yxi}
|y(t)- (T_*-t)^{1/\alpha}\xi_* |=\bigo((T_*-t)^{1/\alpha+\varep})\text{ for some $\varep>0$.}
\eeq
\end{theorem}
\begin{proof}
Regarding the function $f$ that satisfies \eqref{frate}, we observe, for any number $\delta'\in(0,\delta)$, that
\beq \label{smallrate}
|f(t)|\le M'|y(t)|^{1-\alpha+\delta'}\text{ for all $t\in [t_0,T_*)$,}
\eeq 
where
\beq\label{Mprime}
M'=M\max_{t\in[t_0,T_*]}|y(t)|^{\delta-\delta'}\in(0,\infty).
\eeq
Note that we used \eqref{yTz} in \eqref{Mprime}.
Because of property \eqref{smallrate}, we can assume that $\delta<\alpha$ in  \eqref{frate}.
 
With the matrix $A$ and function $H$ in \eqref{AHsim}, they certainly satisfy Assumptions \ref{assumpA} and  \ref{assumpH}. Then  Lemma \ref{estthm} applies and the estimates from above and below for $|y(t)|$ in \eqref{yplus}, and estimate \eqref{fbound} for $|f(t)|$ hold true.

For $t\in( t_0,T_*)$, we have from \eqref{dy2}  that
\beq\label{dy3}
\ddt(|y|^{\alpha})=-\alpha a +\alpha |y|^{\alpha-2}f(t)\cdot y.
\eeq
Integrating equation \eqref{dy3} from $t$ to $t'\in(t,T_*)$, passing to the limit $t'\to T_*^-$, and using \eqref{limzero} give 
\beq\label{ynorm}
|y(t)|^{\alpha}=\alpha a (T_*-t)+g(t), 
\text{ where } g(t)=-\alpha \int_t^{T_*} |y(\tau)|^{\alpha-2}f(\tau)\cdot y(\tau)\d\tau.
\eeq
Hence, for all $t\in[t_0,T_*)$, one has $\alpha a (T_*-t)+g(t)>0$.

Using the Cauchy--Schwarz inequality, \eqref{frate} and the upper bound of $|y(t)|$ in \eqref{yplus}, we estimate 
\begin{align*}
|g(t)|&\le \alpha \int_t^{T_*} |y(\tau)|^{\alpha-1}|f(\tau)|\d\tau
\le \alpha M\int_t^{T_*} |y(\tau)|^{\delta} \d\tau \\
&\le \alpha MC_2^\delta \int_t^{T_*} (T_*-\tau)^{\delta/\alpha}\d\tau.
\end{align*}
We obtain
\beq\label{gC}
|g(t)|\le C_3( T_*- t)^{1+\delta/\alpha}\text{ for all $t\in [t_0, T_*)$, 
where
 $C_3=\frac{\alpha MC_2^\delta}{1+\delta/\alpha}$.}
\eeq 

We consider equation \eqref{basicf} as a linear equation of $y$ with time-dependent coefficient $-a|y(t)|^{-\alpha}$ and forcing function $f(t)$. By the variation of constants formula, we solve for $y(t)$ explicitly as
\begin{align*}
y(t)&=e^{-J(t)}\left(y_0+\int_{t_0}^t e^{J(\tau)}f(\tau)\d\tau\right) \text{ for $t\in [t_0,T_*)$,}
\end{align*}
where
\beq\label{Iform}
J(t)=a\int_{t_0}^t |y(\tau)|^{-\alpha}\d\tau.
\eeq
Using \eqref{ynorm} in \eqref{Iform}, we rewrite $J(t)$ as 
\begin{align*}
J(t)=\int_{t_0}^t \frac{a}{a\alpha (T_*-\tau)+  g(\tau)}\d\tau=J_1(t)+J_2(t),
\end{align*}
where
\begin{align*}
J_1(t)&=\int_{t_0}^t \frac{1}{\alpha (T_*-\tau)}\d\tau
\text{ and }
J_2(t)=\int_{t_0}^t h(\tau) \d\tau,\\
\intertext{ with } 
h(\tau)&=\frac{ -g(\tau)}{\alpha (T_*-\tau)(a\alpha (T_*-\tau)+ g(\tau))}.
\end{align*}

Clearly,
\beqs
J_1(t)=-\frac1{\alpha}\ln (T_*-t)+\frac1\alpha\ln (T_*-t_0).
\eeqs
Therefore,
\beq\label{yJform}
y(t)=\frac{(T_*-t)^{1/\alpha}}{(T_*-t_0)^{1/\alpha}} e^{-J_2(t)}\left(y_0+(T_*-t_0)^{1/\alpha}\int_{t_0}^t \frac{e^{J_2(\tau)}}{(T_*-\tau)^{1/\alpha}}f(\tau)\d\tau\right).
\eeq

Consider the integrand $h(\tau)$ of $J_2(t)$. Taking into account the estimate of $|g(\tau)|$ in \eqref{gC}, 
we assert that, as $\tau\to T_*^-$,
\beq\label{htau}
|h(\tau)|=\bigo(|g(\tau)|(T_*-\tau)^{-2})=\bigo((T_*-\tau)^{-1+\delta/\alpha}). 
\eeq
Thus, 
\beq\label{J2lim} 
\lim_{t\to T_*^-}J_2(t)=\int_{t_0}^{T_*} h(\tau)\d\tau=J_*\in\R,
\eeq 
and
\beq\label{JJh1}
J_2(t)=J_*-h_1(t),\text{ where } h_1(t)=\int_t^{T_*}h(\tau)\d\tau\in \R.
\eeq
From estimate \eqref{htau}, it follows  that
\beq\label{h1o}
|h_1(t)|=\bigo((T_*-t)^{\delta/\alpha}) \text{ as $t\to T_*^-$.}
\eeq

Regarding the integral in  formula \eqref{yJform}, we have, thanks to estimate \eqref{fbound} of $|f(t)|$, that 
\beq\label{eJf}
\frac{|f(t) |}{(T_*-t)^{1/\alpha}}
=\bigo((T_*-t)^{-1+\delta/\alpha})  \text{ as $t\to T_*^-$.}
\eeq
Hence,
\beqs
\lim_{t\to T_*^-} \int_{t_0}^t\frac{e^{J_2(\tau)}}{(T_*-\tau)^{1/\alpha}}f(\tau)\d\tau
=\int_{t_0}^{T_*} \frac{e^{J_2(\tau)}}{(T_*-\tau)^{1/\alpha}}f(\tau)\d\tau
=\eta_*\in \R^n,
\eeqs
and 
\beq\label{ifeeta}
\int_{t_0}^t\frac{e^{J_2(\tau)}}{(T_*-\tau)^{1/\alpha}}f(\tau)\d\tau=\eta_*-\eta(t),
\eeq
where
\beqs
\eta(t)=\int_t^{T_*} \frac{e^{J_2(\tau)}}{(T_*-\tau)^{1/\alpha}}f(\tau)\d\tau\in \R^n.
\eeqs
It follows from \eqref{J2lim} and \eqref{eJf} that
\beq\label{etes}
|\eta(t)| =\bigo((T_*-t)^{\delta/\alpha}) \text{ as $t\to T_*^-$.}
\eeq

Combining \eqref{yJform}, \eqref{JJh1} and \eqref{ifeeta} gives
\beqs
y(t)=\frac{(T_*-t)^{1/\alpha}}{(T_*-t_0)^{1/\alpha}} e^{-J_*+h_1(t)}\left(y_0+(T_*-t_0)^{1/\alpha}(\eta_*-\eta(t))\right)
\text{ for $t\in[t_0, T_*)$.}
\eeqs
Then
\begin{align*}
&y(t)-(T_*-t)^{1/\alpha} e^{-J_*} \left(\frac{y_0}{(T_*-t_0)^{1/\alpha}}+\eta_*\right)\\
&=(T_*-t)^{1/\alpha} e^{-J_*} (e^{h_1(t)}-1)\left(\frac{y_0}{(T_*-t_0)^{1/\alpha}}+\eta_*\right)
 - (T_*-t)^{1/\alpha} e^{-J_*+h_1(t)}\eta(t).
\end{align*}

Let $\xi_*=e^{-J_*}( (T_*-t_0)^{-1/\alpha}y_0+\eta_*)\in\R^n$.
This expression and properties \eqref{h1o}, \eqref{etes} imply, as $t\to T_*^-$,
\begin{align*}
\left|y(t)-(T_*-t)^{1/\alpha} \xi_* \right|
&=\bigo\left((T_*-t)^{1/\alpha} (|e^{h_1(t)}-1|+|\eta(t)|)\right)\\
&=\bigo\left((T_*-t)^{1/\alpha} (|h_1(t)|+|\eta(t)|)\right),
\end{align*}
thus,
\beq \label{yy1}
\left|y(t)-(T_*-t)^{1/\alpha} \xi_* \right|
=\bigo((T_*-t)^{1/\alpha+\delta/\alpha}).
\eeq
Therefore, we obtain the desired estimate \eqref{yxi}. 
Because of the lower bound of $|y(t)|$ in \eqref{yplus}, the vector $\xi_*$ in \eqref{yxi} must be nonzeo. 

We prove property \eqref{axi1} now.
 By the triangle inequality and \eqref{yy1}, one has
  \beq\label{yxinorm1}
\big|(T_*-t)^{-1/\alpha} |y(t)|- |\xi_*| \big|=\bigo((T_*-t)^{\delta/\alpha}).
\eeq
From \eqref{ynorm},
\beq\label{yabs2}
(T_*-t)^{-1/\alpha}|y(t)|= \left(a\alpha+\frac{g(t)}{ T_*-t}\right)^{1/\alpha}.
\eeq
Taking into account estimate \eqref{gC} of $|g(t)|$,  we have from \eqref{yabs2} that, as $t\to T_*^-$,
\beq\label{yxinorm2}
\left| (T_*-t)^{-1/\alpha}|y(t)| -(a\alpha)^{1/\alpha}\right|
=\bigo\left( \frac{|g(t)|}{  T_*-t}\right)
=\bigo( (T_*-t)^{\delta/\alpha}).
\eeq  
From the  two asymptotic estimates \eqref{yxinorm1} and \eqref{yxinorm2}, one must have $|\xi_*|=(a\alpha)^{1/\alpha}$, which proves \eqref{axi1}.
The proof is complete.
\end{proof}
 
\begin{remark}
In the case dimension $n=1$, Theorem \ref{simthm} already proves Theorem \ref{mainthm} for any positive constant $A$ and positive function $H\in \mathcal H_{-\alpha}(\R,\R)$. We justify this fact below. 

Let $y(t)$ be the solution of \eqref{mainode} as in Theorem \ref{mainthm}. 
With $H\in \mathcal H_{-\alpha}(\R,\R)$, we have
$$H(x)=\begin{cases} |x|^{-\alpha}H(1), &  \text{ for $x>0$,}\\
                                   |x|^{-\alpha}H(-1),& \text{ for $x<0$.}
             \end{cases} $$
In general, $H(1)\ne H(-1)$, hence it appears that we do not have equation \eqref{basicf} yet.
However, for our continuous solution $y(t)\ne 0$ on $[t_0,T_*)$, we must have either $y(t)>0$ on $[t_0,T_*)$ or  $y(t)<0$ on $[t_0,T_*)$. Therefore, $y(t)$, in fact, satisfies \eqref{basicf} for all $t\in(t_0,T_*)$, with $a=AH(1)$ or $a=AH(-1)$. Then Theorem \ref{simthm} applies. (As a side note, because  $\mathbb S^0=\{-1,1\}$, property (HC) on $\mathbb S^0$  is automatically satisfied.)
\end{remark}

\section{Solutions when the matrix $A$ is symmetric}\label{symcase}

In this section, we assume  Assumptions \ref{assumpA} and \ref{assumpH} hold and, additionally, the matrix $A$ is symmetric.
Let numbers $T_*>t_0\ge 0$ be given, and let function  $y\in C^1([t_0,T_*),\R^n)$ satisfy \eqref{nz}, \eqref{limzero}, \eqref{mainode}--\eqref{frate}.
For $t\in [t_0,T_*)$, define
\beq\label{quot}
\lambda(t)=\frac{y(t)\cdot Ay(t)}{|y(t)|^2} \text{ and }   
 v(t)=\frac{y(t)}{|y(t)|} .
\eeq
(The quotient $\lambda(t)$ in \eqref{quot} imitates the Dirichlet quotient for the heat equations when $A$ is the negative Laplacian.)
Then $\lambda\in C^1([t_0,T_*),\R)$ and   $v\in C^1([t_0,T_*),\R^n)$. Moreover, one has, 
$|v(t)|=1$ and, thanks to \eqref{xAx},  
 \beq\label{lest}
 \Lambda_1 \le  \lambda(t)\le \Lambda_n\le \|A\| \text{ for all $t\in[ t_0,T_*)$.}
 \eeq 
 
\begin{proposition}\label{lem1}
One has  
$$\lim_{t\to T_*^-} \lambda(t)=\Lambda\in \sigma(A).$$
\end{proposition}
\begin{proof}
For $t\in(t_0,T_*)$, we have
\beq\label{lameq}
\lambda'(t)
=\frac2{|y|^2} y' \cdot Ay-\frac{2(y\cdot Ay)}{|y|^4} y' \cdot y
=\frac2{|y|^2} y' \cdot (Ay-\lambda(t) y).
\eeq
By equation \eqref{mainode}, we write $y'$ as
\beqs
y' =-H(y)(Ay-\lambda(t) y) - \lambda(t) H(y)y+f(t),
\eeqs
and use it in \eqref{lameq} to obtain
\beqs
\lambda'(t)
=-\frac{2H(y)}{|y|^2} |Ay-\lambda(t) y|^2 - \frac{2\lambda(t) H(y)}{|y|^2} y\cdot (Ay-\lambda(t) y)+h(t),
\eeqs
where
\beqs
h(t)=\frac2{|y(t)|^2}f(t)\cdot(Ay(t)-\lambda(t) y(t)).
\eeqs
Because $ y(t)\cdot(Ay(t)-\lambda(t) y(t))=0$, it follows that 
\beq\label{lh}
\lambda'(t)
=-2H(y)|Av-\lambda(t) v|^2 +h(t).
\eeq

Using \eqref{frate}, \eqref{lest}, the fact $|v(t)|=1$,  and then  \eqref{yplus}, we estimate  
\beq\label{oh}
|h(t)|\le 4M\|A\|\cdot |y(t)|^{-\alpha+\delta}\le C_5 (T_*-t)^{-1+\delta/\alpha}
\text{ for all $t\in [t_0,T_*)$,}
\eeq
 where $C_5$ is $4M\|A\| C_2^{-\alpha+\delta}$ if $\delta\ge \alpha$, or
$4M\|A\| C_1^{-\alpha+\delta}$ otherwise. 

For $t,t'\in[t_0,T_*)$ with $t'>t$, integrating equation \eqref{lh} from $t$ to $t'$ gives
\beq\label{llH}
\lambda(t')-\lambda(t)+2\int_t^{t'}H(y(\tau))|Av(\tau)-\lambda(\tau) v(\tau)|^2\d\tau =\int_t^{t'}h(\tau)\d\tau.
\eeq
Thanks to \eqref{oh}, the last integral can be estimated as
\beq\label{hitt}
\left|\int_t^{t'}h(\tau)\d\tau\right|\le \frac{\alpha C_5}{\delta}(T_*-t)^{\delta/\alpha}.
\eeq
By taking the limit superior of \eqref{llH}, as $t'\to T_*^-$, we derive
\beq\label{llH2}
\limsup_{t'\to T_*^-}\lambda(t')\le \lambda(t)+\frac{\alpha C_5}{\delta}(T_*-t)^{\delta/\alpha}<\infty.
\eeq
Then taking the limit inferior of \eqref{llH2}, as $t\to T_*^-$, yields
\beqs
\limsup_{t'\to T_*^-}\lambda(t')\le \liminf_{t\to T_*^-}  \lambda(t).
\eeqs
This and \eqref{lest}  imply
\beq\label{limlam} 
\lim_{t\to T_*^-}\lambda(t)=\Lambda\in [\Lambda_1,\Lambda_n].
\eeq

It remains to be proved that  $\Lambda$ is an eigenvalue of $A$.
Using properties \eqref{hitt} and  \eqref{limlam} in \eqref{llH} and by the Cauchy criterion, as $t,t'\to T_*^-$, we obtain
\beq\label{Hfin}
\int_{t_0}^{T_*} H(y(\tau))|Av(\tau)-\lambda(\tau) v(\tau)|^2\d\tau<\infty.
\eeq

We claim that 
\beq \label{claim} 
\forall\varep\in(0,T_*-t_0),\exists t\in [T_*-\varep,T_*):|Av(t)-\lambda(t)v(t)|<\varep.
\eeq 
Indeed, suppose the claim \eqref{claim} is not true, then 
\beq\label{anticlaim} 
\exists\varep_0\in(0,T_*-t_0),\forall t\in [T_*-\varep_0,T_*): |Av(t)-\lambda(t)v(t)|\ge \varep_0.
\eeq
Combining \eqref{anticlaim} with property \eqref{Hyt}, we have
\beqs
\int_{T_*-\varep_0}^{T_*} H(y(\tau))|Av(\tau)-\lambda(\tau) v(\tau)|^2\d\tau
\ge \int_{T_*-\varep_0}^{T_*} C_3(T_*-\tau)^{-1} \varep_0^2 \d\tau=\infty,
\eeqs
which contradicts \eqref{Hfin}. Hence, the  claim \eqref{claim} is true. 

Thanks to \eqref{claim}, there exists a sequence $(t_j)_{j=1}^\infty\subset [t_0,T_*)$ such that  
\beq \label{Avj}
\lim_{j\to\infty}t_j= T_* 
\text{ and }
\lim_{j\to\infty}|Av(t_j)-\lambda(t_j)v(t_j)|= 0.
\eeq 
The first equation in \eqref{Avj} and \eqref{limlam} imply $\lambda(t_j)\to \Lambda$ as $j\to\infty$. 
Because $v(t_j)\in\mathbb S^{n-1}$ for all $j$, we can extract a subsequence $(v(t_{j_k}))_{k=1}^\infty$, such that $v(t_{j_k})\to \bar v\in\mathbb S^{n-1}$ as $k\to\infty$. 
Combining these limits with the second equation in \eqref{Avj} written with $j=j_k$ and $k\to\infty$ yields
$A\bar v=\Lambda\bar v$.
Therefore, $\Lambda$ is an eigenvalue of $A$.
\end{proof}

From here to the end of this section, $\Lambda$ is the eigenvalue in Proposition \ref{lem1}.

\begin{proposition}\label{lem2}
There is $\varep>0$ such that 
\beq\label{remv}
| (I_n-R_{\Lambda})v(t)|=\bigo((T_*-t)^{\varep}) \text{ as $t\to T_*^-$.}
 \eeq
\end{proposition}
\begin{proof}
If $\sigma(A)=\{\Lambda\}$, then $R_\Lambda={\rm Id}$ and \eqref{remv} is true.
Consider the case $\sigma(A)\ne \{\Lambda\}$.
We calculate
\beqs
 v' =\frac1{|y|}y' -\frac1{|y|^3}(y' \cdot y)y
 =-\frac{H(y)}{|y|}Ay+\frac1{|y|}f(t)
+\frac{H(y)(Ay)\cdot y}{|y|^3}y
 -\frac{f(t)\cdot y}{|y|^3}y.
\eeqs

Define the function $g:[t_0,T_*)\to\R^n$  by
\beqs
g(t)=\frac{1}{|y(t)|}f(t) -\frac{f(t)\cdot y(t)}{|y(t)|^3}y(t).
\eeqs
Then we have
\beq \label{dv}
 v'=-H(y)(Av-\lambda(t) v)+g(t)\text{ for all $t\in (t_0,T_*)$.}
\eeq 
Using property \eqref{frate} of $f(t)$, one can estimate
\beq\label{gMy}
|g(t)|\le 2 M |y(t)|^{-\alpha+\delta}\text{ for all $t\in[t_0,T_*)$.}
\eeq

Let $\lambda_j\in\sigma(A)\setminus\{\Lambda\}$.
Applying $R_{\lambda_j}$ to equation \eqref{dv} and taking the dot product with $R_{\lambda_j}v$ yield
\beq\label{Rnorm}
\frac12 \ddt |R_{\lambda_j}v|^2
 =-H(y)(\lambda_j-\lambda(t))| R_{\lambda_j}v|^2 +R_{\lambda_j} g(t)\cdot R_{\lambda_j}v.
\eeq

Set 
\beq \label{mudef}
\mu=\min \{ |\lambda_j-\Lambda| : 1\le j\le d, \lambda_j\ne  \Lambda\}>0.
\eeq

Applying  Cauchy--Schwarz's inequality, inequality \eqref{Rcontract} to $|R_{\lambda_j}g(t)|$,  estimate \eqref{gMy} for $|g(t)|$, and then Cauchy's inequality, we have
\beqs
|R_{\lambda_j}g(t)\cdot R_{\lambda_j}v|
\le 2M  |y|^{-\alpha+\delta}|R_{\lambda_j}v|
\le \frac\mu4 H(y)| R_{\lambda_j}v|^2 + \frac{4M^2|y|^{-2\alpha+2\delta}}{\mu H(y)}.
\eeqs
Using the first inequality of \eqref{Hyal} to estimate the last $H(y)$ gives 
\beqs
|R_{\lambda_j}g(t)\cdot R_{\lambda_j}v|
\le \frac\mu4 H(y)| R_{\lambda_j}v|^2 + \frac{4M^2}{\mu c_1}|y|^{-\alpha+2\delta}.
\eeqs
Utilizing the estimates in \eqref{yplus} for the norm $|y(t)|$, we obtain, for $t\in [t_0,T_*)$,
\beq\label{abso}
|R_{\lambda_j}g(t)\cdot R_{\lambda_j}v|
\le \frac\mu4 H(y)| R_{\lambda_j}v|^2 +\frac{C_6}2 (T_*-t)^{-1+2\delta/\alpha},
\eeq
where
\beqs
C_6=\frac{8M^2}{\mu c_1}\cdot
\begin{cases} 
   C_2^{-\alpha+2\delta},&\text{ if $\delta\ge \alpha/2$,}\\
   C_1^{-\alpha+2\delta},&\text{ otherwise. }
\end{cases}
\eeqs

Below, $T\in (t_0,T_*)$ is fixed and can be taken sufficiently close to $T_*$ such that
\beq\label{Tlam}
|\lambda(t)-\Lambda|\le \frac\mu4 \text{ for all $t\in[T,T_*)$.}
\eeq

\medskip
\noindent\underline{\textit{Case $\lambda_j>\Lambda$.}} In this case, combining \eqref{Rnorm} and \eqref{abso} yields, for $t\in (t_0,T_*)$,
\begin{align*}
\frac12 \ddt |R_{\lambda_j}v|^2
 &\le -(\lambda_j-\lambda(t)-\frac\mu4) H(y)| R_{\lambda_j}v|^2 +\frac{C_6}2 (T_*-t)^{-1+2\delta/\alpha} .
\end{align*}

By definition \eqref{mudef} of $\mu$ and the choice \eqref{Tlam}, one has,  for all $t\in [T,T_*)$,
\beq\label{lamu} 
\lambda_j-\lambda(t)-\frac\mu4= (\lambda_j-\Lambda)+(\Lambda-\lambda(t))-\frac\mu4
\ge \mu-\frac\mu4-\frac\mu4=\frac{\mu}2.
\eeq
Thus, for $t\in [T,T_*)$,
\beq\label{dRlarge}
 \ddt |R_{\lambda_j}v|^2
 \le - \mu H(y)| R_{\lambda_j}v|^2 +C_6(T_*-t)^{-1+2\delta/\alpha}.
\eeq

Let $t$ and $\bar t$ be any numbers in $[T,T_*)$ with  $t>\bar t$.
 It follows from  \eqref{dRlarge} that 
\beq\label{RHint}
  |R_{\lambda_j}v(t)|^2
 \le e^{-\mu\int_{\bar t}^t H(y(\tau))\d\tau} | R_{\lambda_j}v(\bar t)|^2 + C_6\int_{\bar t}^t e^{-\mu\int_{\tau}^t H(y(s))\d s}(T_*-\tau)^{-1+2\delta/\alpha}\d\tau .
\eeq
With $C_3$ being the positive constant in \eqref{Hyt}, we fix  a number $\theta>0$ such that 
$$\theta\le C_3\text{ and } \theta\mu < 2\delta/\alpha.$$ 
Then 
\beq \label{Hthe}
H(y(t))\ge \theta (T_*-t)^{-1}\text{ for all $t\in [ t_0,T_*)$.}
\eeq 
Utilizing this estimate in \eqref{RHint} gives
\begin{align*}
 & |R_{\lambda_j}v(t)|^2
 \le e^{-\theta\mu\int_{\bar t}^t (T_*-\tau)^{-1}\d\tau} | R_{\lambda_j}v(\bar t)|^2 
 + C_6\int_{\bar t}^t e^{-\theta\mu\int_{\tau}^t (T_*-s)^{-1}\d s}(T_*-\tau)^{-1+2\delta/\alpha}\d\tau \\
  &= \frac{(T_*-t)^{\theta\mu}}{(T_*-\bar t)^{\theta\mu}} | R_{\lambda_j}v(\bar t)|^2 
  + C_6(T_*-t)^{\theta\mu} \int_{\bar t}^t (T_*-\tau)^{-1+2\delta/\alpha-\theta\mu}\d\tau \\
  &= \frac{(T_*-t)^{\theta\mu}}{(T_*-\bar t)^{\theta\mu}} | R_{\lambda_j}v(\bar t)|^2 
    +\frac{C_6(T_*-t)^{\theta\mu}}{2\delta/\alpha-\theta\mu}\left(  (T_*-\bar t)^{2\delta/\alpha-\theta\mu}-(T_*-t)^{2\delta/\alpha-\theta\mu}\right).
\end{align*}
Therefore,
\beq\label{Rlamv1}
  |R_{\lambda_j}v(t)|^2
    \le \left(\frac{| R_{\lambda_j}v(\bar t)|^2}{(T_*-\bar t)^{\theta\mu}}  
    +\frac{C_6(T_*-\bar t)^{2\delta/\alpha-\theta\mu}}{2\delta/\alpha-\theta\mu}\right)(T_*-t)^{\theta\mu} .
\eeq
Having $\bar t=T$ in \eqref{Rlamv1}, we obtain
\beq \label{Rv1}
|R_{\lambda_j}v(t)|=\bigo((T_*-t)^{\theta\mu/2}) \text{ as $t\to T_*^-$.}
\eeq 

\medskip
\noindent\underline{\textit{Case $\lambda_j<\Lambda$.}} Using \eqref{abso} to have a lower bound for  the last term in \eqref{Rnorm}, we have
\begin{align*}
\frac12 \ddt |R_{\lambda_j}v|^2
 &\ge (\lambda(t)-\lambda_j-\frac\mu4) H(y)| R_{\lambda_j}v|^2 -\frac{C_6}2 (T_*-t)^{-1+2\delta/\alpha} .
\end{align*}

Same as \eqref{lamu}, one has, for  $t\in [T,T_*)$, 
$$\lambda(t)- \lambda_j -\frac\mu4= (\lambda(t)-  \Lambda)+(\Lambda-\lambda_j)-\frac\mu4
\ge -\frac\mu4 +\mu -\frac{\mu}4=\frac{\mu}2.$$ 
Hence,
\begin{align*}
\ddt |R_{\lambda_j}v|^2
 &\ge \mu H(y)| R_{\lambda_j}v|^2 -C_6(T_*-t)^{-1+2\delta/\alpha} .
\end{align*}
Then, for any $t,\bar t\in  [T,T_*)$ with $t>\bar t$, one has
\beq\label{eHR}
e^{-\mu\int_{\bar t}^t H(y(\tau))\d\tau}  |R_{\lambda_j}v(t)|^2-  |R_{\lambda_j}v(\bar t)|^2
 \ge -C_6 \int_{\bar t}^t e^{-\mu\int_{\bar t}^\tau H(y(s))\d s}(T_*-\tau)^{-1+2\delta/\alpha} \d\tau.
\eeq

Note from \eqref{Hthe} that $\int_{\bar t}^{T_*} H(y(\tau))\d\tau=\infty$, and from \eqref{Rcontract} that $ |R_{\lambda_j}v(t)|\le |v(t)|=1$. Then
\beqs
\lim_{t\to T_*^-} e^{-\mu\int_{\bar t}^t H(y(\tau))\d\tau}  |R_{\lambda_j}v(t)|^2=0.
\eeqs
Letting $t\to T_*^-$ in \eqref{eHR} and using \eqref{Hthe}  yield
\begin{align*}
|R_{\lambda_j}v(\bar t)|^2 
 &\le C_6 \int_{\bar t}^{T_*} e^{-\mu\int_{\bar t}^\tau H(y(s))\d s}(T_*-\tau)^{-1+2\delta/\alpha} \d\tau\\
&\le C_6 \int_{\bar t}^{T_*} \frac{(T_*-\tau)^{\theta\mu}}{(T_*-\bar t)^{\theta\mu}}(T_*-\tau)^{-1+2\delta/\alpha} \d\tau
=  \frac{C_6}{\theta\mu+2\delta/\alpha} (T_*-\bar t)^{2\delta/\alpha}.
\end{align*}
Therefore, we obtain
\beq\label{Rv2} 
|R_{\lambda_j}v(\bar t)| =\bigo((T_*-\bar t)^{\delta/\alpha})\text{ as $\bar t\to T_*^-$.}
\eeq 

\medskip
We estimate $|(I_n-R_{\Lambda})v(t)|$ now.
We have 
\beq\label{InR} 
|(I_n-R_{\Lambda})v(t)|=\Big|\sum_{1\le j\le d, \lambda_j\ne  \Lambda } R_{\lambda_j}v(t)\Big|
\le \sum_{1\le j\le d, \lambda_j\ne  \Lambda } |R_{\lambda_j}v(t)|.
\eeq
In the last sum in \eqref{InR}, we estimate $|R_{\lambda_j}v(t)|$ for all  $\lambda_j>\Lambda$ by \eqref{Rv1}, and  estimate  $|R_{\lambda_j}v(t)|$  for all   $\lambda_j<\Lambda$ by \eqref{Rv2}.
This results in the desired estimate   \eqref{remv} for $|(I_n-R_{\Lambda})v|$, with $\varep=\min\{\theta\mu/2,\delta/\alpha\}=\theta\mu/2$.
\end{proof}

We derive from Proposition \ref{lem2} more specific estimates for $y(t)$.
Let $\varep>0$ be as in Proposition \ref{lem2}.  On the one hand, we have \beqs
|(I_n-R_\Lambda)y(t)|=|y(t)|\cdot  |(I_n-R_\Lambda)v(t)|.
\eeqs
Together with \eqref{yplus} and \eqref{remv}, it yields
\beq\label{remy}
|(I_n-R_\Lambda)y(t)|=\bigo((T_*-t)^{1/\alpha+\varep})\text{ as $t\to T_*^-$.}
\eeq
On the other hand, by the triangle inequality and \eqref{yplus}, one has
\begin{align*}
|R_\Lambda y(t)|&\le |y(t)|+|(I_n-R_\Lambda)y(t)|\le C_2 (T_*-t)^{1/\alpha}+|(I_n-R_\Lambda)y(t)|,\\
|R_\Lambda y(t)|&\ge |y(t)|-|(I_n-R_\Lambda)y(t)|\ge  C_1 (T_*-t)^{1/\alpha}-|(I_n-R_\Lambda)y(t)|.
\end{align*}
Combining these inequalities with estimate \eqref{remy} for $|(I_n-R_\Lambda)y(t)|$, we deduce that
there exist numbers $T_0\in[t_0,T_*)$ and $C_7,C_8>0$ such that
\beq\label{RLy}
C_7 (T_*-t)^{1/\alpha}\le |R_\Lambda y(t)|\le C_8 (T_*-t)^{1/\alpha} \text{ for all $t\in[T_0,T_*)$.}
\eeq

\begin{proposition}\label{lem4}
There exists a unit vector $v_*\in\R^n$ such that
\beq\label{RLv}
|R_{\Lambda}v(t)-v_*|=\bigo((T_*-t)^{\varep})\text{ as $t\to T_*^-$ for some $\varep>0$.}
\eeq
\end{proposition}
\begin{proof}
Let $\varep_0>0$ be such that \eqref{remv} holds for $\varep=\varep_0$.
Then one has
\beq\label{oneR}
\big| 1-|R_{\Lambda}v(t)| \big|
=\big| |v(t)|-|R_{\Lambda}v(t)|\big|\le |v(t)-R_{\Lambda}v(t)|=\bigo((T_*-t)^{\varep_0}).
\eeq

Let $T_0$ be as in \eqref{RLy}. Recall that $C_4$ is the positive constant in \eqref{Hyt}.
We  fix a number $\varep_1>0$  such that
\beq\label{cep1}
C_4\varep_1<\delta/\alpha.
\eeq

Thanks to Proposition \ref{lem1},  there is $T\in [T_0,T_*)$  such that 
\beqs
|\lambda(t)-\Lambda|\le \varep_1 \text{ for all $t\in[T,T_*)$.}
\eeqs

Note from \eqref{RLy} that $R_{\Lambda}v(t)\ne 0$ for all $t\in[T,T_*)$.
Applying $R_\Lambda$ to equation \eqref{dv} yields, for  $t\in(t_0,T_*)$,
\beq \label{dRv}
\ddt R_{\Lambda}v
 =-H(y)(\Lambda-\lambda(t)) R_{\Lambda}v +R_\Lambda g(t).
\eeq
Then,  for $t\in[T,T_*)$,
\beq\label{RLveq}
\ddt |R_{\Lambda}v|
=\frac{1}{|R_{\Lambda}v|} \left(\ddt R_{\Lambda}v\right)\cdot R_{\Lambda}v
 =-H(y)(\Lambda-\lambda(t))| R_{\Lambda}v| +g_1(t),
\eeq
where
\beqs
g_1(t)=\frac{R_\Lambda  g(t)\cdot R_{\Lambda}v(t)}{|R_{\Lambda}v(t)|}.
\eeqs

Solving for solution $ |R_{\Lambda}v(t)|$ by the variation of constants formula from the  differential equation \eqref{RLveq} gives, for $\bar t,t\in[T,T_*)$ with $t>\bar t$, 
\beqs
 |R_{\Lambda}v(t)|=e^{-\int_{\bar t}^t H(y(\tau))(\Lambda-\lambda(\tau))\d\tau } \left(| R_{\Lambda}v(\bar t) |
 + \int_{\bar t}^t e^{\int_{\bar t}^\tau H(y(s))(\Lambda-\lambda(s))\d s}  g_1(\tau) \d\tau \right).
\eeqs
It yields
\beq \label{hl0}
\begin{aligned}
&\int_{\bar t}^t H(y(\tau))(\Lambda-\lambda(\tau))\d\tau\\
&=\ln \left(| R_{\Lambda}v(\bar t) |
 + \int_{\bar t}^t e^{\int_{\bar t}^\tau H(y(s))(\Lambda-\lambda(s))\d s}  g_1(\tau) \d\tau \right)
   - \ln |R_{\Lambda}v(t)|.
\end{aligned}
\eeq

We have from \eqref{Rcontract}, \eqref{yplus} and \eqref{gMy} that
\beq\label{gg}
|g_1(t)|\le|R_\Lambda  g(t)| \le  |g(t)|\le C_9 (T_*-t)^{-1+\delta/\alpha} \text{ for all $t\in [T,T_*)$,}
\eeq
where $C_9$ is $2MC_2^{-\alpha+\delta}$ if $\delta\ge \alpha$, and is $2MC_1^{-\alpha+\delta}$ otherwise.
By \eqref{gg} and \eqref{Hyt}, we have, for $\tau\in[\bar t,T_*)$,
\begin{align*}
e^{\int_{\bar t}^\tau H(y(s))(\Lambda-\lambda(s))\d s} | g_1(\tau) |
&\le e^{\int_{\bar t}^\tau C_4\varep_1 (T_*-s)^{-1}\d s} C_9(T_*-\tau)^{-1+\delta/\alpha}\\
&=C_9(T_*-\bar t)^{C_4\varep_1} (T_*-\tau)^{-1+\delta/\alpha-C_4\varep_1}.
\end{align*}
Thanks to this and \eqref{cep1}, 
\beq\label{eta0}
\begin{aligned}
\lim_{t\to T_*^-}\int_{\bar t}^t e^{\int_{\bar t}^\tau H(y(s))(\Lambda-\lambda(s))\d s}  g_1(\tau) \d\tau
&=\int_{\bar t}^{T_*} e^{\int_{\bar t}^\tau H(y(s))(\Lambda-\lambda(s))\d s}  g_1(\tau) \d\tau\\
&=\eta(\bar t)\in \R.
\end{aligned}
\eeq
Note that 
\beq\label{eta}
|\eta(\bar t)|\le  C_9(T_*-\bar t)^{C_4\varep_1} \int_{\bar t}^{T_*} (T_*-\tau)^{-1+\delta/\alpha-C_4\varep_1}\d\tau
=\frac{C_9}{\delta/\alpha-C_4\varep_1} (T_*-\bar t)^{\delta/\alpha}.
\eeq

Passing to the limit as $t\to T_*^-$ in \eqref{hl0}, and using the fact $|R_\Lambda v(t)|\to 1$, thanks to \eqref{oneR}, we have
\beq\label{hl1}
\int_{\bar t}^{T_*} H(y(\tau))(\Lambda-\lambda(\tau))\d\tau= \ln (| R_{\Lambda}v(\bar t) |
 + \eta(\bar t))\in\R.
\eeq
By \eqref{hl1}, we can define, for $t\in[T,T_*)$,
$$h(t)=\int_{t}^{T_*} H(y(\tau))(\Lambda-\lambda(\tau))\d\tau\in\R.$$
We rewrite \eqref{hl1} for $\bar t=t$ as
\beqs
h(t)=\ln (| R_{\Lambda}v(t) | + \eta(t))=\ln (1+(| R_{\Lambda}v(t) | -1)+ \eta(t)).
\eeqs
With this expression and properties \eqref{oneR} and \eqref{eta}, we have, as $t\to T_*^-$, 
\beq\label{hl2}
| h(t) |
=\bigo(\big|| R_{\Lambda}v(t) | -1\big|+ |\eta(t)|)
=\bigo((T_*-t)^{\varep_0}+(T_*-t)^{\delta/\alpha})
=\bigo((T_*-\bar t)^{\varep_2}),
\eeq
where $\varep_2=\min\{\varep_0,\delta/\alpha\}$.

Solving for $R_{\Lambda}v(t)$ from \eqref{dRv} by the variation of constants formula, one has
\beq\label{RvH}
R_{\Lambda}v(t)
 =e^{-\int_{\bar t}^t H(y(\tau))(\Lambda-\lambda(\tau))\d\tau } \left( R_{\Lambda}v(\bar t) 
 + \int_{\bar t}^t e^{\int_{\bar t}^\tau H(y(s))(\Lambda-\lambda(s))\d s}  R_{\Lambda}g(\tau)\d\tau \right) .
\eeq
Using the same arguments as those from \eqref{gg} to \eqref{eta} with $R_\Lambda g(\tau)$ replacing $g_1(\tau)$, we obtain, similar to \eqref{eta0} and \eqref{eta}, that
\begin{align*}
\lim_{t\to T_*^-} \int_{\bar t}^{t} e^{\int_{\bar t}^\tau H(y(s))(\Lambda-\lambda(s))\d s}  R_{\Lambda}g(\tau)\d \tau 
&= \int_{\bar t}^{T_*} e^{\int_{\bar t}^\tau H(y(s))(\Lambda-\lambda(s))\d s}  R_{\Lambda}g(\tau) \d\tau\\
&=X(\bar t)\in\R^n
\end{align*}
for all $\bar t\in[T,T_*)$, and  
 \beq \label{oX}
 |X(\bar t)|=\bigo((T_*-\bar t)^{\varep_2})\text{ as $\bar t\to T_*^-$.}
\eeq

Taking $t\to T_*^-$ in \eqref{RvH} gives
\beqs
\lim_{t\to T_*^-}R_{\Lambda}v(t)=v_*\eqdef e^{-h(\bar t)}( R_{\Lambda}v(\bar t) +X(\bar t))\in\R^n.
\eeqs
Note that
\beqs
X(t)=\int_{t}^{T_*} e^{h(t)-h(\tau)}  R_{\Lambda}g(\tau) \d\tau.
\eeqs
Using $h(t)$, $X(t)$ and $v_*$, we rewrite \eqref{RvH} as
\beqs
R_{\Lambda}v(t)
=e^{h(t)-h(\bar t)} \left (R_{\Lambda}v(\bar t)+X(\bar t)-\int_t^{T_*} e^{h(\bar t)-h(\tau)} R_{\Lambda}g(\tau) \d\tau\right)
=e^{h(t)}v_*-X(t).
\eeqs
Thus,
\beqs
|R_{\Lambda}v(t)-v_*|
\le |e^{h(t)}-1|\cdot |v_*| + |X(t)|.
\eeqs
Using \eqref{hl2} and \eqref{oX},  we deduce, as $t\to T_*^-$,
\beqs
|R_{\Lambda}v(t)-v_*|
=\bigo(|h(t))|+|X(t)|)=\bigo((T_*-t)^{\varep_2}).
\eeqs
Therefore, we obtain the desired estimate \eqref{RLv}.
Because of \eqref{RLv} and \eqref{oneR}, we have $|v_*|=1$.
The proof is complete.
\end{proof}

Some immediate consequences of \eqref{remv} and \eqref{RLv} are
\beq\label{Rvlim}
\lim_{t\to T_*^-} R_{\Lambda}v(t)=\lim_{t\to T_*^-} v(t)=v_*,
\eeq 
and
\beq\label{vlim}
|v(t)-v_*|=\bigo((T_*-t)^{\varep}) \text{ as $t\to T_*^-$ for some $\varep>0$.}
\eeq

\section{Proofs of the main theorems}\label{gencase}

We prove  Theorems \ref{mainthm} and \ref{mainthm2} in this section.

\begin{proof}[Proof of Theorem \ref{mainthm}]
Since $H$ satisfies Assumption \ref{Hcond} then, thanks to Lemma \ref{Hcts}, it is positive, continuous on $\R^n\setminus \{0\}$, and has property (HC)  on $\R^n\setminus \{0\}$.

\medskip
\noindent\textit{Case 1.} Consider the case $A$ is symmetric first. We use the same notation as in Section \ref{symcase}. 

The estimate \eqref{newIR} already comes from \eqref{remy}. 
We prove  \eqref{newRy} next. Applying $R_\Lambda$ to equation \eqref{mainode}, we have
\beq\label{Ry1}
(R_\Lambda y)' =-\Lambda  H(y)R_\Lambda y+R_\Lambda f(t).
\eeq

Let $v_*$ be the unit vector in Proposition \ref{lem4}, and $\varep_0>0$ be such that \eqref{remv}, \eqref{remy}, \eqref{RLv}, and \eqref{vlim} hold for $\varep=\varep_0$.
We rewrite $H(y)$ on the right-hand side of \eqref{Ry1} as
\beqs
H(y)=|y|^{-\alpha} H(v)=|R_\Lambda y|^{-\alpha} H(v_*)+g_0(t),
\eeqs
where
\begin{align*}
g_0(t)&=|y(t)|^{-\alpha} (H(v(t))-H(v_*))+(|y(t)|^{-\alpha} -|R_\Lambda y(t)|^{-\alpha})H(v_*)\\
&=|y(t)|^{-\alpha} \big\{ H(v(t))-H(v_*)+(1 -|R_\Lambda v(t)|^{-\alpha})H(v_*)\big\}.
\end{align*}
Then 
\beq\label{Rygood}
(R_\Lambda y)' 
= -\Lambda H(v_*) |R_\Lambda y|^{-\alpha}R_\Lambda y+f_0(t),
\eeq
where
$f_0(t)=-\Lambda g_0(t)R_\Lambda y(t) +R_\Lambda f(t)$.

We estimate $|g_0(t)|$ first. We combine inequality \eqref{FHder} in Definition \ref{HCdef} applied to 
$F=H$, $E=\mathbb S^{n-1}$, $x_0=v_*$  and  $x=v(t)$ for $t$ sufficiently close to $T_*$, with estimate \eqref{vlim}.
Then there exists a number $\gamma>0$ such that, as $t\to T_*^-$,
\beq\label{Hvv}
|H(v(t))-H(v_*)|=\bigo(|v(t)-v_*|^\gamma) =\bigo((T_*-t)^{\gamma\varep_0}).
\eeq
It is elementary to see  $|s^{-\alpha}-1|=\bigo(|s-1|)$ as $s\to 1$.
By taking $s=|R_\Lambda v(t)|$,  which goes to $1$ as $t\to T_*^-$  thanks to \eqref{Rvlim}, and using estimate  \eqref{oneR}, we derive
\beq\label{1Ral}
\big|1 -|R_\Lambda v(t)|^{-\alpha}\big|
=\bigo\left(\big |1-|R_\Lambda v(t)|\big| \right)=\bigo((T_*-t)^{\varep_0})  \text{ as $t\to T_*^-$.}
\eeq
Combining \eqref{Hvv}, \eqref{1Ral} with \eqref{yplus}, we obtain
\beq\label{g1}
|g_0(t)|=\bigo(|y(t)|^{-\alpha} ((T_*-t)^{\gamma\varep_0}+(T_*-t)^{\varep_0}))=\bigo((T_*-t)^{-1+\varep_1})
\eeq
as $t\to T_*^-$, where $\varep_1=\varep_0\min\{1,\gamma\}$.

We estimate $|f_0(t)|$ now. As $t\to T_*^-$, we have from  \eqref{yplus} and \eqref{fbound}  that
\beq\label{Ry5}
|R_\Lambda y(t)|=\bigo((T_*-t)^{1/\alpha})\text{ and }
|R_\Lambda f(t)|=\bigo((T_*-t)^{1/\alpha-1+\delta/\alpha}).
\eeq
Combining \eqref{g1} and \eqref{Ry5} gives, as $t\to T_*^-$, 
\beqs
|f_0(t)|=\bigo((T_*-t)^{-1+\varep_1}(T_*-t)^{1/\alpha}+(T_*-t)^{1/\alpha-1+\delta/\alpha})
=\bigo((T_*-t)^{1/\alpha-1+\varep_2/\alpha}),
\eeqs
where $\varep_2=\min\{\varep_1\alpha ,\delta\}$.
By the virtue of the lower bound of $|R_\Lambda y(t)|$ in \eqref{RLy}, we actually have
\beqs
|f_0(t)|=\bigo(|R_\Lambda y(t)|^{1-\alpha+\varep_2}).
\eeqs

Fix a number $t_0'\in[t_0,T_*)$ such that $R_\Lambda y(t)\ne 0$ and 
$|f_0(t)|\le M_0|R_\Lambda y(t)|^{1-\alpha+\varep_2}$ for all $t\in[t_0',T_*)$, where $M_0$ is a positive constant.   Of course, one already has $R_\Lambda y(T_*)=0$.

We apply Theorem \ref{simthm} to solution $R_\Lambda y(t)$ of equation \eqref{Rygood} on the interval $[t_0',T_*)$. Specifically, $R_\Lambda y(t)$ satisfies equation \eqref{basicf} on $(t_0',T_*)$ with constant $a=\Lambda H(v_*) $ and $f=f_0$. Then there exists a nonzero vector $\xi_*\in\R^n$ such that 
\beq\label{Rxi3}
|R_\Lambda y(t)-(T_*-t)^{1/\alpha}\xi_* |=\bigo((T_*-t)^{1/\alpha+\varep_3})\text{ for some $\varep_3>0$,}
\eeq
and 
\beq\label{Hv1}
|\xi_*|=(\alpha \Lambda H(v_*) )^{1/\alpha} .
\eeq
The desired statement  \eqref{newRy} immediately follows from \eqref{Rxi3}.

Because 
$$\xi_*=\lim_{t\to T_*^-} (T_*-t)^{-1/\alpha}R_\Lambda y(t),$$
 by \eqref{newRy}, and the fact $\xi_*\ne 0$, we have $\xi_*\in R_\Lambda(\R^n)\setminus\{0\}$. Hence, $\xi_*$ is an eigenvector of $A$ associated with $\Lambda$.

Next, we prove \eqref{mainest}. Writing 
$$y(t)-(T_*- t)^{1/\alpha}\xi_*=(I_n-R_\Lambda)y(t)+(R_\Lambda y(t)- (T_*- t)^{1/\alpha}\xi_*),$$
 and using  the estimate  \eqref{remy} with $\varep=\varep_0$, and estimate \eqref{Rxi3} yield
\beqs
|y(t)-(T_*- t)^{1/\alpha}\xi_*| = \bigo((T_*- t)^{1/\alpha+\varep_0}+(T_*- t)^{1/\alpha+\varep_3}).
\eeqs
This implies \eqref{mainest} with $\varep=\min\{\varep_0,\varep_3\}$.

Finally, we prove \eqref{xiHA}. Let $w(t)=(T_*- t)^{-1/\alpha}y(t)$ and write 
$v(t)=w(t)/|w(t)|$.
Passing $t\to T_*^-$ and noticing that $v(t)\to v_*$ and $w(t)\to \xi_*$, thanks to \eqref{Rvlim} and \eqref{mainest},
we obtain
\beq\label{vxi} 
v_*=\xi_*/|\xi_*|
\eeq 
Then it follows from  \eqref{Hv1}, the fact $H$ is positively homogeneous of degree $-\alpha$, and relation \eqref{vxi}   that
$$1=\alpha \Lambda H(v_*)  |\xi_*|^{-\alpha}=\alpha \Lambda H(|\xi_*|v_* )
=\alpha \Lambda H(\xi_*).$$
Hence, we obtain \eqref{xiHA}. 
This completes the proof for the case of symmetric matrix $A$.

\medskip
\noindent\textit{Case 2.} Consider the case $A$ is not symmetric. Let $A_0$ and $S$ be as in \eqref{Adiag}.
Same as \eqref{zSy}, we set $z(t)=Sy(t)$ on $[t_0,T_*]$. Then $z(t)$ satisfies equation \eqref{zeq} with $\widetilde H$ and $\widetilde f$ defined in \eqref{HRz}. 
One can verify the following facts. 
\begin{itemize}
\item $z(t)\ne 0$ for  $t\in [t_0,T_*)$ and $z(T_*)=0$.
\item $\widetilde H\in \mathcal H_{-\alpha}(\R^n)$ and, thanks to parts \ref{HH1} and \ref{HH5} of Lemma \ref{Hcts}, 
$\widetilde H>0$ on $\mathbb S^{n-1}$ and $\widetilde H$ has property (HC) on $\mathbb S^{n-1}$.
\item Thanks to \eqref{fzMtil},  $\widetilde f(t)$ and $z(t)$  satisfy condition \eqref{frate} with the same numbers $\alpha,\delta,t_0,T_*$, and constant $\widetilde M$ in place of $M$.
\end{itemize} 

We apply the results already established in Case 1 above to the solution $z(t)$ of equation \eqref{zeq}. Note that $A_0$ replaces $A$ and $\widehat R_{\lambda_j}$  replaces $R_{\lambda_j}$.
Then there exist an eigenvalue $\Lambda$ of $A_0$ and an eigenvector $\xi_0$ of $A_0$ associated  with $\Lambda$ such that 
\beq\label{hatRz}
\begin{split}
|(I_n-\widehat R_\Lambda)z(t)|=\bigo((T_*-t)^{1/\alpha+\varep}),\\
|\widehat R_\Lambda z(t)-(T_*-t)^{1/\alpha}\xi_0 |=\bigo((T_*-t)^{1/\alpha+\varep})
\end{split}
\eeq
for some number $\varep>0$, and
\beq\label{A0xi} 
\alpha \Lambda \widetilde H(\xi_0)=1.
\eeq

Let $\xi_*=S^{-1}\xi_0$. 
Then $\Lambda$ is an eigenvalue of $A$ and  $\xi_*$ is an eigenvector of $A$ associated  with $\Lambda$. 
We rewrite \eqref{hatRz} as 
\beqs 
\begin{split}
|S(I_n-R_\Lambda)y(t)|=\bigo((T_*-t)^{1/\alpha+\varep}),\\
|S( R_\Lambda y(t)-(T_*-t)^{-1/\alpha}\xi_* )|=\bigo((T_*-t)^{1/\alpha+\varep})
\end{split}
\eeqs
which imply \eqref{newIR} and \eqref{newRy}. 
By  \eqref{newIR},  \eqref{newRy} and the triangle inequality, we obtain \eqref{mainest} in the same way as in Case 1.

Finally, \eqref{xiHA} follows from \eqref{A0xi} and the relation $\widetilde H(\xi_0)=H(\xi_*)$.
The proof of Theorem \ref{mainthm} is complete.
\end{proof}

\begin{proof}[Proof of Theorem \ref{mainthm2}]
 Set $f(t)=G(t,y(t))$ for $t\in[t_0,T_*)$.
Thanks to \eqref{limzero}, the is a number $t'_0\in[t_0,T_*)$ such that
\beqs
|y(t)|\le r_* \text{ for all } t\in[t'_0,T_*).
\eeqs
This property and \eqref{Fcond} imply 
\beqs
|f(t)|\le c_* |y(t)|^{1-\alpha+\delta}\text{ for all }t\in[t'_0,T_*).
\eeqs
Thus, $f(t)$ satisfies condition \eqref{frate} with $t'_0$ replacing $t_0$. Applying Theorem \ref{mainthm} to the interval $[t'_0,T_*)$ in place of $[t_0,T_*)$, we obtain the statements of Theorem \ref{mainthm2}.
\end{proof}

\section{Examples and applications}\label{examples}

\subsection{Examples}\label{egs}
We give some examples for the function $H$ in Assumptions \ref{assumpH} and \ref{Hcond}. 
For simplicity, we consider the case dimension $n=2$. One can easily generalize them for any higher dimension $n$.
\begin{enumerate}[label=\rnum]
\item  For $x=(x_1,x_2)\in \R^2\setminus\{0\}$, let
\begin{align*}
H_1(x)&=(x_1^4+5x_2^4)^{-3}, \\
H_2(x)&=\left[ (3|x_1|^{3/2}+|x_2|^{3/2})^{1/3}+(2|x_1|^{5/3}+7|x_2|^{5/3})^{3/10} \right]^{-1/8}.
\end{align*}
Then $H_1$ is in $\mathcal H_{-3/4}(\R^2,\R)$ while  $H_2$ is in $\mathcal H_{-1/16}(\R^2,\R)$. Both functions belong to $C^1(\R^2\setminus\{0\})$. Hence, they have property (HC) on $\mathbb S^1$ with the same power $\gamma=1$ in \eqref{FHder}.
 \item  
Another example is 
\beqs
H(x)=\frac{\sqrt{|x_1|}+\sqrt{|x_2|}}{|x|} \text{ for $x=(x_1,x_2)\in \R^2\setminus\{0\}$.}
\eeqs
Then $H$ belongs to $\mathcal H_{-1/2}(\R^2,\R)$, is positive on $\mathbb S^1$ and has property (HC) on  $\mathbb S^1$ with the same power $\gamma=1/2$ in \eqref{FHder}. Unlike the previous two examples, this function $H$ is not in  $C^1(\R^2\setminus\{0\})$.
\end{enumerate}
In fact, the function $H$ can be very complicated, see similar examples in \cite[Example 5.7]{H7} and \cite[Section 6]{CaHK1}.

\subsection{Applications}\label{bioeg}

We give an application to a population model \cite{KaKa2016} in mathematical biology.
Consider an inhomogeneous population composed of individuals with different death rates.
This population consists of $n$ clones, each has the size $y_i(t)$, for $i=1,2,\ldots,n$,  at time $t$ with the death rate $k_i>0$, where the constants $k_i$ are mutually distinct.
The model \cite[Equations (3.1) and (3.2)]{KaKa2016} is
\beq\label{biosys}
 y_i' = - k_i y_i g(N) \quad \text{ for $i=1,2,\ldots,n$, }
\eeq
where 
\beq \label{total}
N=N(t)\eqdef y_1(t)+y_2(t)+\ldots+y_n(t)\text{  is the total population size,}
\eeq 
and $g(s):(0,\infty)\to (0,\infty)$ is an appropriate function.
 Note that  the integral form of $N(t)$ in  \cite[Equation (3.2)]{KaKa2016} becomes the finite sum in \eqref{total}.  

Below, we consider the case 
\beq\label{gpower}
g(s)=k s^{-\alpha} \text{ for any $s>0$, where $k>0$ and $\alpha>0$ are some constants. }
\eeq 
  This generalizes the consideration $\alpha=1$ in \cite[Equation (4.2)]{KaKa2016}.
  
Define the matrix $A$ and function $H(x)$, for $x=(x_1,\ldots,x_n)\in \R^n\setminus\{0\}$, as follows
$$A=k\, {\rm diag}[k_1,k_2,\ldots,k_n]\text{  and } H(x)=(|x_1|+|x_2|+\ldots+|x_n|)^{-\alpha}.$$
 
Denote $\mathcal C=\{ x=(x_1,\ldots,x_n)\in \R^n:x_i\ge 0 \text{ for }i=1,2,\ldots,n\}$.
With $y=(y_1,\ldots,y_n)$, if $y(t)\in \mathcal C\setminus\{0\}$ is a solution of the system \eqref{biosys}, \eqref{total}, \eqref{gpower}, then it is a solution of system \eqref{yReq} with $G\equiv 0$, i.e.,
\beq\label{reduced}
y'=-H(y)Ay.
\eeq

Note that the matrix $A$ is symmetric and satisfies Assumption \ref{assumpA}. The function $H$ clearly belongs to $H_{-\alpha}(\R^n,\R)$.

\medskip \noindent
\textbf{Claim.}\textit{ The function $H(x)$ has property  (HC) on $\mathbb S^{n-1}$ with the same power $\gamma=\min\{1,\alpha\}$ in \eqref{FHder}.}

Thus, $H$ satisfies Assumption \ref{Hcond}.

 \begin{proof}[Proof of the Claim.] Denote the $\ell^1$-norm $|x|_{\ell^1}=|x_1|+|x_2|+\ldots+|x_n|$. Then there exists a constant $c_0\ge 1$ such that
  $$c_0^{-1}|x|\le |x|_{\ell^1}\le c_0 |x| \text{ for all }x\in \R^n.$$
  For $x,y\in \mathbb S^{n-1}$, we have both norms $|x|_{\ell^1}$ and $|y|_{\ell^1}$ belong to the interval  $[c_0^{-1},c_0] $, and, hence,
  \begin{align*}
  |H(x)-H(y)|=\frac{\left| |y|_{\ell^1}^\alpha-|x|_{\ell^1}^\alpha\right| } {|x|_{\ell^1}^\alpha |y|_{\ell^1}^\alpha}
  \le c_0^{2\alpha} \left| |y|_{\ell^1}^\alpha-|x|_{\ell^1}^\alpha\right| .
  \end{align*}

When $0<\alpha\le 1$,  utilizing the inequality $ \left| |y|_{\ell^1}^\alpha-|x|_{\ell^1}^\alpha\right| \le  \left| |y|_{\ell^1}-|x|_{\ell^1}\right|^\alpha$, and applying the triangle inequality, one obtains
  \begin{align*}
  |H(x)-H(y)|  \le c_0^{2\alpha}  |y-x|_{\ell^1}^\alpha 
  \le c_0^{3\alpha}  |y-x|^\alpha.
  \end{align*}

When $\alpha>1$, applying the Mean Value Theorem to the function $s\mapsto s^\alpha$ and values $s_1=|y|_{\ell^1}$, $s_2=|x|_{\ell^1}$ both belonging to the interval $[c_0^{-1},c_0]$ yields the existence of a number $C_0>0$ depending on $c_0$ and $\alpha$ such that 
$$ \left| |y|_{\ell^1}^\alpha-|x|_{\ell^1}^\alpha\right| \le C_0 \left| |y|_{\ell^1}-|x|_{\ell^1}\right|.$$ 
Then applying the triangle inequality again, we deduce
  \begin{align*}
  |H(x)-H(y)|\le c_0^{2\alpha}  C_0 \left| |y|_{\ell^1}-|x|_{\ell^1}\right|
  \le c_0^{2\alpha}  C_0  |y-x|_{\ell^1}
  \le c_0^{2\alpha+1} C_0 |y-x|.
  \end{align*}  
  
  Therefore, the Claim is true.
  \end{proof}

\begin{theorem}\label{bio2}
Let $y_0\in\mathcal C\setminus\{0\}$ be sufficiently small, then there exist a number $T_*>0$ and a function $y\in C^1([0,T_*),\mathcal C\setminus\{0\})$  such that $y(t)$ satisfies  \eqref{biosys}, \eqref{total}, \eqref{gpower} for all $t\in(0,T_*)$, $y(0)=y_0$, and \eqref{limzero} holds.
\end{theorem}
\begin{proof}
Applying Theorem \ref{mustdie} to equation \eqref{reduced} and $t_0=0$,  we obtain a number $T_*>0$ and a function $y\in C^1([0,T_*),\R^n\setminus\{0\})$ such that $y(t)$ satisfies  equation \eqref{reduced} for all $t\in(0,T_*)$, $y(0)=y_0$, and \eqref{limzero} holds. Since $y(t)\ne 0$, for all $t\in[0,T_*)$, we have, for each $i=1,2,\ldots,n$,
\beq
y_i(t)=y_i(0)e^{-kk_i\int_0^t H(y(\tau))\d\tau}.
\eeq
Together with the fact $y(0)\in \mathcal C\setminus\{0\}$, this implies $y(t)\in \mathcal C\setminus\{0\}$, and hence $y(t)$ is a solution of \eqref{biosys}, \eqref{total}, \eqref{gpower} for all $t\in(0,T_*)$.
\end{proof}
 
  Let $\{e_i:i=1,2,\ldots,n\}$ denote the standard canonical basis of $\R^n$.

\begin{theorem}\label{bio1}
 Let $T_*>0$ and $y\in C^1([0,T_*),\mathcal C\setminus\{0\})$ be such that $y(t)$ satisfies  \eqref{biosys}, \eqref{total}, \eqref{gpower} for all $t\in(0,T_*)$ and \eqref{limzero} holds.
 Then there is an integer $i\in[1,n]$ and a number $\varep>0$ such that
 \beq\label{applest}
 \left|y(t)-(\alpha k k_i)^{1/\alpha}(T_*-t)^{1/\alpha}e_i\right|=\bigo\left((T_*-t)^{1/\alpha+\varep}\right) \text{ as } t\to T_*^-.
 \eeq
\end{theorem}
\begin{proof}
Since $y(t)\in \mathcal C\setminus\{0\}$, the function $y$ in fact is a solution of \eqref{reduced} on $(0,T_*)$ with the extinction time $T_*$. By the virtue of Theorem \ref{mainthm2} applied to equation \eqref{reduced} and $t_0=0$, there exist a number $\varep>0$, an eigenvalue $\Lambda$ of $A$ and a corresponding eigenvector $\xi_*$, such that
\beq\label{aest1}
\left|y(t)-(T_*-t)^{1/\alpha}\xi_*\right|=\bigo\left((T_*-t)^{1/\alpha+\varep}\right) \text{ as } t\to T_*^-.
\eeq 
Clearly, $\Lambda=k k_i$ for some $1\le i\le n$, and $\xi_*=Ke_i$ for some number $K\ne 0$.
 Multiplying  estimate \eqref{aest1} by $ (T_*-t)^{-1/\alpha}$ and using the $i$th coordinate, one derives 
 $$K=\lim_{t\to T_*^-} y_i(t) (T_*-t)^{-1/\alpha},$$
 which implies $K\ge 0$. 
 Moreover, thanks to \eqref{xiHA}, $|K|^\alpha=\alpha k k_i$. Therefore, $K=(\alpha k k_i)^{1/\alpha}$ and we obtain \eqref{applest} from \eqref{aest1}.
 \end{proof}

 Note in this demonstration that the system \eqref{biosys} is simpler than \eqref{yReq} and \eqref{mainode} which were theoretically studied in the previous sections.

\begin{remark} The following final remarks are in order.

\begin{enumerate}[label=\rnum]
\item 
There is another totally different approach to the local properties of solutions of ODE based on the Poincar\'e--Dulac normal form \cite{ArnoldODEGeo,BibikovBook,LefschetzBook}. It has been generalized and  developed by many and for so long, see the books \cite{BrunoBook1989,BrunoBook2000,KFbook2013},  recent papers such as \cite{Bruno2004,Bruno2008c,Bruno2012,Bruno2018}, and references therein. Our approach is relatively new and only recently used to explore different classes of equations and problems in ODE.
For  more comparisons between the other approach and ours, see \cite[Remark 5.8]{H7} and \cite[Remark 6.14]{CaHK1}.

\item It is an open problem whether the solutions of \eqref{yReq} admit an asymptotic expansion similar to those in \cite{FS84a,CaHK1} near the extinction time. 
Taking some indications from \cite{CaHK1}, we expect, in the case the answer is affirmative,  that our result \eqref{mainest} will play an important role in its proof.
\end{enumerate}
\end{remark}

\appendix

\section{}

\begin{proof}[Proof of Lemma \ref{Hcts}]
Denote $E=\R^n\setminus \{0\}$.
For any $x\in E$, we can write $F(x)=|x|^{-\alpha}F(x/|x|)$. Hence, part \ref{HH1} is obvious.
For parts \ref{HH2}--\ref{HH4}, the proofs are similar to the proof of \cite[Lemma 5.1]{H7},  and \textit{``the verification of Assumption 5.2 for $\widetilde H$"} in the proof of \cite[Theorem 5.3]{H7}. We present the key arguments here.

Let $x,\xi\in E$. Then
\begin{align}
 |F(x)-F(\xi)|
 &=  \left ||x|^{-\alpha}F(x/|x|)-|\xi|^{-\alpha}F(\xi/|\xi|)\right| \notag\\
 &\le |x|^{-\alpha}\left | F(x/|x|)-F(\xi/|\xi|)\right| +\left||x|^{-\alpha}- |\xi|^{-\alpha}\right| \cdot\big|F(\xi/|\xi|)\big| \label{Htri}
\end{align}
Using  inequality \eqref{Htri} and the fact that functions $x\in E\mapsto x/|x|$ and $x\in E\mapsto |x|^{-\alpha}$ are $C^1$-functions, we can prove parts \ref{HH2} and \ref{HH3}.

We prove part \ref{HH4} now. Suppose $F$ has property (HC) on  $\mathbb S^{n-1}$.
By part \ref{HH3},  $F$ has property (HC) on  $E$. Clearly, $\varphi$ is a continuous function on $E$.
Let $\xi$ be any vector in $E$. Consider $x\in E$ sufficiently close to $\xi$. 
As $x\to \xi$, we have $\varphi(x)\to \varphi(\xi)$. Using inequality \eqref{FHder} for function $F$ and $x_0:=\varphi(\xi)\in E$, $x:=\varphi(x)\in E$ with constant $C$ and power $\gamma$, and then inequality \eqref{FHder} again for function $\varphi$ and $x_0:=\xi\in E$, $x\in E$ with constant $C'$ and power $\gamma'$, we have
\beqs
|F(\varphi(x))-F(\varphi(\xi))|\le C|\varphi(x)-\varphi(\xi)|^\gamma\le C\,{C'}^\gamma |x-\xi|^{\gamma\gamma'}.
\eeqs
Therefore, the function $F\circ\varphi$ has property (HC) on $E$.

Part \ref{HH5} is a direct consequence of part \ref{HH4} with $\varphi(x)=Kx$. We omit the details.
\end{proof}

\medskip
\noindent\textbf{Data availability.} 
No new data were created or analyzed in this study.

\medskip
\noindent\textbf{Funding.} No funds were received for conducting this study. 

\medskip
\noindent\textbf{Conflict of interest.}
There are no conflicts of interests.

\bibliography{paperbaseall}{}
\bibliographystyle{plain}

 \end{document}